\documentclass[a4paper,12pt]{article}
\usepackage{amsmath,amsthm}
\usepackage{authblk}
\usepackage[style=numeric-comp,maxbibnames=99,bibencoding=utf8,firstinits=true]{biblatex}
\usepackage{booktabs}
\usepackage{cancel}
\usepackage{colortbl}
\usepackage{comment}
\usepackage[hmargin={25mm,25mm},vmargin={30mm,35mm}]{geometry}
\usepackage{loglogslopetriangle}
\usepackage{multirow}
\usepackage{nicefrac}
\usepackage{paralist}
\usepackage[colorlinks, allcolors=blue]{hyperref}
\usepackage{pgfplots}
\usepackage{subcaption}
\usepackage[compact,small]{titlesec}
\usepackage{tikz}
\usepackage{tikz-cd}
\usepackage{newtxmath}
\usepackage{newtxtext}

\bibliography{brinkman-poly}

\graphicspath{{./figures}}

\newcommand{\email}[1]{\href{mailto:#1}{#1}}

\allowdisplaybreaks

\newtheorem{theorem}{Theorem}
\newtheorem{proposition}[theorem]{Proposition}
\newtheorem{lemma}[theorem]{Lemma}

\theoremstyle{remark}
\newtheorem{remark}[theorem]{Remark}
\theoremstyle{definition}

\newcommand{\st}{\,:\,}
\newcommand{\Real}{\mathbb{R}}

\DeclareRobustCommand{\bvec}[1]{\boldsymbol{#1}}
\newcommand{\uvec}[1]{\underline{\bvec{#1}}}

\newcommand{\cvec}[1]{\bvec{\mathcal{#1}}}

\DeclareRobustCommand{\btens}[1]{\boldsymbol{#1}}
\newcommand{\ctens}[1]{\bvec{\mathcal{#1}}}

\newcommand{\Id}{\btens{I}_d}

\DeclareMathOperator{\card}{card}

\DeclareMathOperator{\tr}{tr}

\newcommand{\GRAD}{\boldsymbol{\nabla}}
\newcommand{\GRADs}{\GRAD_{\rm s}}
\newcommand{\DIV}{\boldsymbol{\nabla}{\cdot}}
\newcommand{\VDIV}{\boldsymbol{\nabla}{\cdot}}

\newcommand{\Hdiv}[1]{\bvec{H}(\operatorname{div};#1)}
\newcommand{\HZdiv}[1]{\bvec{H}_0(\operatorname{div};#1)}

\newcommand{\compl}{{\rm c}}

\newcommand{\Poly}[1]{\mathcal{P}^{#1}}
\newcommand{\vPoly}[1]{\cvec{P}^{#1}}
\newcommand{\tPoly}[1]{\ctens{P}^{#1}}
\newcommand{\Goly}[1]{\cvec{G}^{#1}}
\newcommand{\cGoly}[1]{\cvec{G}^{\compl,#1}}

\newcommand{\lproj}[2]{\pi_{\mathcal{P},#2}^{#1}}
\newcommand{\vlproj}[2]{\bvec{\pi}_{\cvec{P},#2}^{#1}}
\newcommand{\tlproj}[2]{\btens{\pi}_{\cvec{P},#2}^{#1}}
\newcommand{\Gproj}[2]{\bvec{\pi}_{\cvec{G},#2}^{#1}}

\newcommand{\elements}[1]{\mathcal{T}_{#1}}
\newcommand{\faces}[1]{\mathcal{F}_{#1}}

\newcommand{\FT}{\faces{T}}

\newcommand{\normal}{\bvec{n}}

\newcommand{\Mh}{\mathcal{M}_h}
\newcommand{\Th}{\elements{h}}
\newcommand{\Fh}{\faces{h}}

\newcommand{\boundary}{{\rm b}}
\newcommand{\Fhb}{\faces{h}^\boundary}

\newcommand{\norm}[2]{\|#2\|_{#1}}
\newcommand{\seminorm}[2]{|#2|_{#1}}
\newcommand{\vvvert}{\vert\kern-0.25ex\vert\kern-0.25ex\vert}

\newcommand{\term}{\mathfrak{T}}

\newcommand{\Cf}[1]{C_{{\rm f},#1}}
\newcommand{\hCf}[1]{\widehat{C}_{{\rm f},#1}}

\newcommand{\Uh}{\uvec{U}_h^k}
\newcommand{\UT}{\uvec{U}_T^k}
\newcommand{\UhD}{\uvec{U}_{h,0}^k}
\newcommand{\Ph}{P_h^k}
\newcommand{\Xdiv}[1]{\underline{\bvec{X}}_{\operatorname{div},#1}^k}

\newcommand{\Ih}{\uvec{I}_h^k}
\newcommand{\IT}{\uvec{I}_T^k}

\newcommand{\stokes}{\mathrm{S}}
\newcommand{\darcy}{\mathrm{D}}

\newcommand{\GT}{\btens{G}_T^k}
\newcommand{\PST}{\bvec{P}_{\stokes,T}^{k+1}}
\newcommand{\DT}{D_T^k}
\newcommand{\PDT}{\bvec{P}_{\darcy,T}^k}
\newcommand{\tPDT}{\widetilde{\bvec{P}}_{\darcy,T}^k}

\newcommand{\Err}{\mathcal{E}_h^k}
\newcommand{\Errsto}{\mathcal{E}_{\stokes,h}^k}
\newcommand{\Errdar}{\mathcal{E}_{\darcy,h}^k}
\newcommand{\Errc}{\mathcal{E}_{\mathrm{c},h}^k}
\newcommand{\Errrhs}{\mathcal{E}_{\mathrm{rhs},h}^k}

\newcommand{\tv}[1]{\langle#1\rangle}

\newcommand{\regT}{\lambda_T}

\begin{document}

\title{A polytopal method for the Brinkman problem robust in all regimes}
\author[1]{Daniele A. Di Pietro}
\affil[1]{IMAG, Univ Montpellier, CNRS, Montpellier, France, \email{daniele.di-pietro@umontpellier.fr}}
\author[2]{J\'{e}r\^{o}me Droniou}
\affil[2]{School of Mathematics, Monash University, Melbourne, Australia, \email{jerome.droniou@monash.edu}}

\maketitle

\begin{abstract}
  In this work we develop a discretisation method for the Brinkman problem that is uniformly well-behaved in all regimes (as identified by a local dimensionless number with the meaning of a friction coefficient) and supports general meshes as well as arbitrary approximation orders.
  The method is obtained combining ideas from the Hybrid High-Order and Discrete de Rham methods, and its robustness rests on a potential reconstruction and stabilisation terms that change in nature according to the value of the local friction coefficient.
  We derive error estimates that, thanks to the presence of cut-off factors, are valid across all the regimes and provide extensive numerical validation.
\end{abstract}



\section{Introduction}

The Brinkman problem governs the flow of a viscous fluid in an inhomogeneous material where fractures, bubbles, or channels are present within a porous matrix.
Mathematically, this problem translates into a system of partial differential equations with saddle-point structure which can be regarded as a superposition of the Stokes and Darcy systems.
As pointed out in \cite{Mardal.Tai.ea:02}, the construction of finite element approximations that are uniformly well-behaved across the entire range of (Stokes- or Darcy-dominated) regimes is not straightforward; a representative, but by far not exhaustive, list of references is \cite{Burman.Hansbo:05,Burman.Hansbo:07,Juntunen.Stenberg:10,Konno.Stenberg:11,Anaya.Gatica.ea:15,Evans.Hughes:13,Caceres.Gatica.ea:17,Vacca:18,Araya.Harder.ea:17,Zhao.Chung.ea:20}.
In \cite{Botti.Di-Pietro.ea:18}, we introduced a numerical method for the Brinkman problem on matching simplicial meshes and derived what appears to be the first error estimate accounting for the local regime through a dimensionless number which can be interpreted as a friction coefficient.
Thanks to the presence of cutoff factors, this error estimate holds in all situations, including the Stokes problem as well as the singular limit corresponding to the pure Darcy problem.

In this work, we provide a positive answer to an open question left in the above reference, namely whether similar robustness features and error estimates can be obtained on general polytopal meshes.
As for the original method of \cite{Botti.Di-Pietro.ea:18}, the discretisation of the Stokes term is inspired by Hybrid High-Order (HHO) methods \cite{Di-Pietro.Ern.ea:14,Di-Pietro.Ern:15,Di-Pietro.Droniou:20} while, for the Darcy and forcing terms, a novel construction inspired by discrete de Rham methods \cite{Di-Pietro.Droniou.ea:20,Di-Pietro.Droniou:21*1} (see also \cite{Di-Pietro.Ern:17} for an antecedent) replaces the one based on the Raviart--Thomas--N\'ed\'elec space \cite{Raviart.Thomas:77,Nedelec:80}.
The first central element in this construction is a discrete vector potential that changes in nature depending on the value of the local friction coefficient.
The other key ingredient are regime-dependent stabilisation terms.
Thanks to these novel tools, we are able to derive a robust estimate of the adjoint error for the discrete divergence, which is the pivot result for the extension of the techniques of \cite{Botti.Di-Pietro.ea:18} to polytopal meshes.
The resulting error estimate, stated in Theorem~\ref{thm:error.estimate} below, is valid on the entire range of values of the local friction coefficient, from $0$ (pure Stokes) to $+\infty$ (pure Darcy).

The rest of the work is organised as follows.
In Section~\ref{sec:setting} we briefly recall the continuous and discrete settings.
In Section~\ref{sec:scheme} we formulate the numerical scheme and state the main stability and convergence results.
Extensive numerical validation of these results on a variety of meshes and regimes for analytical solutions is provided in Section~\ref{sec:numerical.tests}, where a more physical three-dimensional test case is also considered.
Finally, the proofs of the main results are collected in Section~\ref{sec:analysis}.


\section{Setting}\label{sec:setting}

\subsection{Continuous problem}

Let $\Omega\subset\Real^d$, $d\in\{2,3\}$, denote a bounded connected open polytopal (i.e., polygonal if $d=2$ and polyhedral if $d=3$) domain with boundary $\partial\Omega$.
For the sake of simplicity, and without loss of generality, we assume that $\Omega$ has unit diameter.
Let two functions $\mu:\Omega\to\Real$ and $\nu:\Omega\to\Real$ be given.
In what follows, we assume that there exist real numbers $\underline{\mu},\overline{\mu}$, and $\overline{\nu}$ such that, almost everywhere in $\Omega$,
\begin{equation}\label{eq:mu.nu:bounds}
  0<\underline{\mu}\le\mu\le\overline{\mu},\qquad
  0\le\nu\le\overline{\nu}.
\end{equation}
Let $\bvec{f}:\Omega\to\Real^d$ and $g:\Omega\to\Real$ denote volumetric source terms.
The Brinkman problem reads:
Find the velocity $\bvec{u}:\Omega\to\Real^d$ and the pressure $p:\Omega\to\Real$ such that
\begin{subequations}\label{eq:strong}
  \begin{alignat}{2}
    \label{eq:strong:momentum}
    -\VDIV(\mu\GRAD\bvec{u}) + \nu\bvec{u} + \GRAD p &= \bvec{f} &\qquad&\text{in $\Omega$},
    \\ \label{eq:strong:mass}
    \DIV\bvec{u} &= g &\qquad&\text{in $\Omega$},
    \\ \label{eq:bc}
    \bvec{u} &= \bvec{0} &\qquad&\text{on $\partial\Omega$},
    \\ \label{eq:zero.mean.pressure}
    \int_\Omega p &= 0.
  \end{alignat}
\end{subequations}
A few simplifications are made to make the exposition more compact while retaining all the difficulties related to the robustness across the entire range of values of $\mu$ and $\nu$.
First of all, in \eqref{eq:strong:momentum} we have considered a viscous term expressed in terms of the full gradient instead of its symmetric part $\GRADs$.
The modifications to replace $\GRAD$ with $\GRADs$ are standard in the HHO literature; see, e.g., \cite{Di-Pietro.Ern:15,Botti.Di-Pietro.ea:19*3} and \cite[Chapter~7]{Di-Pietro.Droniou:20}.
Second, we assume henceforth that both $\mu$ and $\nu$ are piecewise constant on a polytopal partition $P_\Omega$ of the domain.
The extension to coefficients that vary smoothly inside each element, and are possibly full tensors, is also standard; see, in particular, \cite[Section~4.2]{Di-Pietro.Droniou:20}.

\subsection{Discrete setting}

\subsubsection{Mesh and notation for inequalities up to a constant}

We consider polytopal meshes $\Mh\coloneq\Th\cup\Fh$ matching the geometrical requirements detailed in \cite[Definition 1.4]{Di-Pietro.Droniou:20}, with $\Th$ set of elements and $\Fh$ set of faces.
To avoid dealing with jumps of the problem coefficients $\mu$ and $\nu$ inside mesh elements, we additionally assume that $\Th$ is compatible with $P_\Omega$, meaning that, for each $T\in\Th$, there exists $\omega\in P_\Omega$ such that $T\subset\omega$.
We then set $\mu_T\coloneq \mu_{|T}$ and $\nu_T\coloneq \nu_{|T}$ for all $T\in\Th$, noticing that these constant values are uniquely defined in each element.
For any $Y\in\Mh$, we denote by $h_Y$ its diameter, so that $h=\max_{T\in\Th}h_T>0$.
For every mesh element $T\in\Th$, we denote by $\FT$ the subset of $\Fh$ containing the faces that lie on the boundary $\partial T$ of $T$.
For any mesh face $F\in\Fh$, we fix once and for all a unit normal vector $\normal_F$ and, for any mesh element $T\in\Th$ such that $F\in\FT$, we let $\omega_{TF}\in\{-1,+1\}$ denote the orientation of $F$ relative to $T$, selected so that $\omega_{TF}\normal_F$ points out of $T$.
Boundary faces lying on $\partial\Omega$  are collected in the set $\Fhb$.

Our focus being on the $h$-convergence analysis, we assume that $\Mh$ belongs to a sequence of refined polytopal meshes that is regular in the sense of \cite[Definition 1.9]{Di-Pietro.Droniou:20}.
This implies, in particular, that the number of faces of each mesh element is bounded from above by an integer independent of $h$; see \cite[Lemma 1.12]{Di-Pietro.Droniou:20}.

From this point on, $a \lesssim b$ means $a\le Cb$ with $C$ only depending on $\Omega$, the mesh regularity parameter, and the polynomial degree $k$ of the scheme defined in Section~\ref{sec:scheme}.
We stress that this means, in particular, that $C$ is independent of the problem parameters $\mu$ and $\nu$.

\subsubsection{Polynomial spaces}

Given $Y\in\Th\cup\Fh$ and an integer $m\ge 0$, we denote by $\Poly{m}(Y)$ the space spanned by the restriction to $Y$ of $d$-variate polynomials of total degree $\le m$.
The symbols $\vPoly{m}(Y;\Real^d)$ and $\tPoly{m}(Y;\Real^{d\times d})$ respectively denote the sets of vector- and tensor-valued functions over $Y$ whose components are in $\Poly{m}(Y)$.
For $T\in\Th$, we will need the following direct decomposition of $\vPoly{m}(T;\Real^d)$ (see, e.g., \cite[Corollary 7.4]{Arnold:18}):
\[
  \vPoly{m}(T;\Real^d)
  = \Goly{m}(T) \oplus \cGoly{m}(T),
  \]
  with
  \begin{equation}\label{eq:Goly.cGoly}
    \text{%
      $\Goly{m}(T)\coloneq\GRAD\Poly{m+1}(T)$\quad and\quad
      $\cGoly{m}(T)\coloneq\begin{cases}
      (\bvec{x} - \bvec{x}_T)^\bot\Poly{m-1}(T) & \text{if $d=2$},
      \\
      (\bvec{x} - \bvec{x}_T)\times\vPoly{m-1}(T;\Real^3) & \text{if $d=3$},
      \end{cases}$
    }
  \end{equation}
where $\bvec{x}_T$ is a point such that $T$ is star-shaped with respect to a ball centered at $\bvec{x}_T$ and of radius $r_T$ such that $h_T\lesssim r_T$ and, in the case $d=2$, for any $\bvec{v}\in\Real^2$ we denote by $\bvec{v}^\bot$ the vector obtained rotating $\bvec{v}$ by $-\frac\pi2$ radians.
Given a polynomial (sub)space $\mathcal{X}^m(Y)$ on $Y\in\Th\cup\Fh$, the corresponding $L^2$-orthogonal projector is denoted by $\pi_{\mathcal{X},Y}^m$.
Boldface fonts will be used when the elements of $\mathcal{X}^m(Y)$ are vector-valued.
The set of broken polynomials of total degree $\le m$ on the mesh is denoted by $\Poly{m}(\Th)$, and the corresponding $L^2$-orthogonal projector by $\lproj{m}{h}$.

\subsubsection{Local friction coefficient}

The regime inside each mesh element $T\in\Th$ is identified by the following dimensionless number, which can be interpreted as a friction coefficient:
\begin{equation}\label{eq:CfT}
  \Cf{T}\coloneq\frac{\nu_T h_T^2}{\mu_T}.
\end{equation}
Elements for which $\Cf{T}<1$ are in the Stokes-dominated regime, while elements for which $\Cf{T}\ge1$ are in the Darcy-dominated regime.
The values $\Cf{T} = 0$ and $\Cf{T} = +\infty$ correspond to pure Stokes and pure Darcy, respectively.
Notice that $\Cf{T} = +\infty$ is a singular limit which, despite requiring to modify the continuous formulation \eqref{eq:strong}, can be handled seamlessly by the method developed in the next section; see Remark~\ref{rem:pure.darcy} below.


\section{A robust numerical scheme for the Brinkman problem}\label{sec:scheme}

\subsection{Spaces}

Let an integer $k\ge 0$ be fixed.
We define the following HHO space:
\begin{multline*}
  \Uh\coloneq\Big\{
  \uvec{v}_h = \big(
  (\bvec{v}_T)_{T\in\Th}, (\bvec{v}_F)_{F\in\Fh}
  \big):
  \\
  \text{%
    $\bvec{v}_T\in\vPoly{k}(T;\Real^d)$ for all $T\in\Th$
    and $\bvec{v}_F\in\vPoly{k}(F;\Real^d)$ for all $F\in\Fh$%
  }
  \Big\}.
\end{multline*}
The meaning of the polynomial components in $\Uh$ is provided by the interpolator $\Ih:\bvec{H}^1(\Omega;\Real^d)\to\Uh$ such that, for all $\bvec{v}\in\bvec{H}^1(\Omega;\Real^d)$,
\begin{equation}\label{eq:def.Ih}
  \Ih\bvec{v}\coloneq\big(
  (\vlproj{k}{T}\bvec{v})_{T\in\Th},
  (\vlproj{k}{F}\bvec{v})_{F\in\Th},
  \big)\in\Uh,
\end{equation}
where it is understood that $L^2$-orthogonal projectors are applied to restrictions or traces as needed.
The restrictions of $\Uh$, $\uvec{v}_h\in\Uh$, and $\Ih$ to a mesh element $T$, respectively denoted by $\UT$, $\uvec{v}_T\in\UT$, and $\IT$, are obtained collecting the components attached to $T$ and its faces.

In what follows, given a logical proposition $P$, we denote by $\tv{P}$ its truth value such that
\begin{equation}\label{eq:tv}
  \tv{P} \coloneq \begin{cases}
    0 & \text{if $P$ is false,}
    \\
    1 & \text{if $P$ is true.}
  \end{cases}
\end{equation}
We define the following $L^2$-like product in $\Uh$:
For all $(\uvec{w}_h,\uvec{v}_h)\in\Uh\times\Uh$,
\begin{equation}\label{eq:prod.U}
  \begin{gathered}
    (\uvec{w}_h, \uvec{v}_h)_{\bvec{U},h}
    \coloneq\sum_{T\in\Th}(\uvec{w}_T, \uvec{v}_T)_{\bvec{U},T}%
    \quad\text{with}\quad%
    \\
    (\uvec{w}_T, \uvec{v}_T)_{\bvec{U},T}
    \coloneq\regT\int_T\bvec{w}_T\cdot\bvec{v}_T
    + h_T\sum_{F\in\FT}\tv{\text{$\Cf{T}<1$ or $F\not\in\Fhb$}}\int_F\bvec{w}_F\cdot\bvec{v}_F,
  \end{gathered}
\end{equation}
where $\regT$ is a factor depending only on the regularity of the element $T$ (but independent of its diameter $h_T$) and selected so as to balance out the element and face contributions in $(\cdot,\cdot)_{\bvec{U},T}$ (see Section \ref{sec:numerical.tests}); in particular, we have $1\lesssim \lambda_T\lesssim 1$.
The corresponding local and global seminorms are obtained setting, for $\bullet\in\Th\cup\{h\}$,
\begin{equation}\label{eq:norm.U}
  \norm{\bvec{U},\bullet}{\uvec{v}_\bullet}\coloneq (\uvec{v}_\bullet, \uvec{v}_\bullet)_{\bvec{U},\bullet}^{\nicefrac12}.
\end{equation}
The following boundedness property of the interpolator in the $\norm{\bvec{U},h}{{\cdot}}$-norm follows from the definition of this norm along with the uniform boundedness of the $L^2$-orthogonal projectors $\vlproj{k}{Y}$, $Y\in\Mh$, and continuous trace inequalities (cf.\ \cite[Lemma~1.31]{Di-Pietro.Droniou:20}):
For all $T\in\Th$ and all $\bvec{v}\in\bvec{H}^1(T;\Real^d)$,
\begin{equation}\label{eq:boundedness:IT:norm.U}
  \norm{\bvec{U},T}{\IT\bvec{v}}\lesssim
  \norm{\bvec{L}^2(T;\Real^d)}{\bvec{v}} + h_T\seminorm{\bvec{H}^1(T;\Real^d)}{\bvec{v}}.
\end{equation}

The velocity and pressure spaces, respectively incorporating the boundary and zero-average conditions, are
\[
\UhD\coloneq\left\{
\uvec{v}_h\in\Uh\st\text{%
  $\bvec{v}_F = \bvec{0}$ for all $F\in\Fhb$
}
\right\},\qquad
\Ph\coloneq\Poly{k}(\Th)\cap L^2_0(\Omega),
\]
where, as usual, $L^2_0(\Omega)=\left\{q\in L^2(\Omega)\,:\,\int_\Omega q=0\right\}$.

\begin{remark}[Boundary degrees of freedom]
  Note that the degrees of freedom on the boundary faces of a vector in $\Uh$ may not be controlled by the seminorms $\norm{\bvec{U},\bullet}{{\cdot}}$. This is, however, not an issue as the final problem will be set on $\UhD$ (see also Remark \ref{rem:pure.darcy} for the handling of boundary values in the limiting case of the pure Darcy problem).
\end{remark}

\subsection{Viscous term}

Let $T\in\Th$ be fixed.
For the discretisation of the viscous term, we define the \emph{discrete gradient} $\GT:\UT\to\tPoly{k}(T;\Real^{d\times d})$ and the \emph{Stokes potential} $\PST:\UT\to\vPoly{k+1}(T;\Real^d)$ such that, for all $\uvec{v}_T\in\UT$,
\begin{equation}\label{eq:GT}
  \int_T\GT\uvec{v}_T:\btens{\tau}
  = -\int_T\bvec{v}_T\cdot\VDIV\btens{\tau}
  + \sum_{F\in\FT}\omega_{TF}\int_F\bvec{v}_F\cdot\btens{\tau}\normal_F
  \qquad
  \forall\btens{\tau}\in\tPoly{k}(T;\Real^{d\times d}),
\end{equation}
and
\begin{equation}\label{eq:PST}
  \GRAD\PST\uvec{v}_T = \Gproj{k}{T}\GT\uvec{v}_T,\qquad
  \int_T\PST\uvec{v}_T = \int_T\bvec{v}_T,
\end{equation}
with $\Gproj{k}{T}$ applied to tensor-valued fields acting row-wise.
Likewise, in the formulas above, $\VDIV$ and $\GRAD$ are understood to act row-wise.

The Stokes term in \eqref{eq:strong:momentum} is discretised through the bilinear form $a_{\mu,h}:\Uh\times\Uh\to\Real$ such that, for all $(\uvec{w}_h, \uvec{v}_h)\in\Uh\times\Uh$,
\begin{equation}\label{eq:a.mu.h}
  a_{\mu,h}(\uvec{w}_h, \uvec{v}_h)\coloneq\sum_{T\in\Th}\mu_T a_{\stokes,T}(\uvec{w}_T, \uvec{v}_T),
\end{equation}
where, for all $T\in\Th$,
\begin{equation}\label{eq:a.stokes.T}
  a_{\stokes,T}(\uvec{w}_T, \uvec{v}_T)
  \coloneq\int_T\GT\uvec{w}_T:\GT\uvec{v}_T
  + \frac{\min(1,\Cf{T}^{-1})}{h_T^2}(
  \uvec{w}_T - \IT\PST\uvec{w}_T,
  \uvec{v}_T - \IT\PST\uvec{v}_T
  )_{\bvec{U},T}.
\end{equation}

\begin{remark}[Stabilisation]
  In the above local bilinear form, the stabilisation contribution penalises separately the element difference $\bvec{w}_T - \vlproj{k}{T}\PST\uvec{w}_T$ and face differences $\bvec{w}_F - \vlproj{k}{F}\PST\uvec{w}_T$, $F\in\FT$, respectively through the first and second term in $(\cdot,\cdot)_{\bvec{U},T}$ (cf.~\eqref{eq:prod.U}).
  This is akin to what is often done in Virtual Element methods (see, e.g., \cite{Beirao-da-Veiga.Brezzi.ea:13} for an introduction), whereas the original HHO stabilisation of \cite{Di-Pietro.Ern.ea:14} involves penalisations at faces only; see \cite[Section 2.1.4]{Di-Pietro.Droniou:20} for a broader discussion on this subject. 
  Notice, however, the special treatment of boundary faces in the bilinear form $(\cdot,\cdot)_{\bvec{U},T}$, which are present only for Stokes-dominated regions: this ensures the applicability and robustness of the scheme in the pure Darcy limit (see Remark \ref{rem:pure.darcy} below), and justifies the choice of this particular bilinear form to define the stabilisation terms for both the Stokes and Darcy terms (see \eqref{eq:a.darcy.T} below).
\end{remark}

We define the following induced seminorms: For all $\uvec{v}_h\in\Uh$,
\begin{equation}\label{eq:norm.mu.stokes}
  \norm{\mu,h}{\uvec{v}_h}\coloneq a_{\mu,h}(\uvec{v}_h,\uvec{v}_h)^{\nicefrac12}
  \quad\text{and}\quad
  \text{
    $\norm{\stokes,T}{\uvec{v}_T}\coloneq a_{\stokes,T}(\uvec{v}_T,\uvec{v}_T)^{\nicefrac12}$ for all $T\in\Th$.
  }
\end{equation}

\begin{lemma}[Norm equivalence]\label{lem:hho.stabilisation}
  For all $T\in\Th$ and all $\uvec{v}_T\in\UT$, it holds
  \begin{equation}\label{eq:seminorm.1.T:upper.bound}
    \norm{\stokes,T}{\uvec{v}_T}^2
    \lesssim
    \norm{\btens{L}^2(T;\Real^{d\times d})}{\GRAD\bvec{v}_T}^2
    + \frac{1}{h_T}\sum_{F\in\FT}\norm{\bvec{L}^2(F;\Real^d)}{\bvec{v}_T - \bvec{v}_F}^2.
  \end{equation}
  Assuming, moreover, $\Cf{T}< 1$, we also have
  \begin{equation}\label{eq:seminorm.1.T:lower.bound}
    \norm{\btens{L}^2(T;\Real^{d\times d})}{\GRAD\bvec{v}_T}^2
    + \frac{1}{h_T}\sum_{F\in\FT}\norm{\bvec{L}^2(F;\Real^d)}{\bvec{v}_T - \bvec{v}_F}^2
   \lesssim \norm{\stokes,T}{\uvec{v}_T}^2.
  \end{equation}
    
\end{lemma}

\begin{proof}
  For the sake of brevity, we only prove \eqref{eq:seminorm.1.T:lower.bound}.
  The proof of \eqref{eq:seminorm.1.T:upper.bound} hinges on similar arguments, together with the fact that $\min(1,\Cf{T}^{-1})\le 1$, and is left to the reader.
  Taking $\btens{\tau} = \GRAD\bvec{v}_T$ in the definition \eqref{eq:GT} of $\GT$, integrating by parts the first term in the right-hand side, and using Cauchy--Schwarz and discrete trace inequalities (see \cite[Lemma~1.32]{Di-Pietro.Droniou:20}) as in the proof of \cite[Eq.~(2.25)]{Di-Pietro.Droniou:20}, we get, after simplifying and raising to the square,
  \begin{equation}\label{eq:est.norm.grad.vT}
    \norm{\btens{L}^2(T;\Real^{d\times d})}{\GRAD\bvec{v}_T}^2
    \lesssim\norm{\btens{L}^2(T;\Real^{d\times d})}{\GT\uvec{v}_T}^2
    + h_T^{-1}\sum_{F\in\FT}\norm{\bvec{L}^2(F;\Real^d)}{\bvec{v}_T - \bvec{v}_F}^2.
  \end{equation}
  To estimate the second term, for any $F\in\FT$, we insert $\pm(\vlproj{k}{T}\PST\uvec{v}_T - \vlproj{k}{F}\PST\uvec{v}_T)$ and use triangle inequalities to get
  \begin{equation}\label{eq:est.vT-vF}
    \begin{aligned}
      &h_T^{-1}\norm{\bvec{L}^2(F;\Real^d)}{\bvec{v}_T - \bvec{v}_F}^2
      \\
      &\quad\lesssim
      h_T^{-1}\norm{\bvec{L}^2(F;\Real^d)}{\bvec{v}_T - \vlproj{k}{T}\PST\uvec{v}_T}^2
      + h_T^{-1}\norm{\bvec{L}^2(F;\Real^d)}{\bvec{v}_F - \vlproj{k}{F}\PST\uvec{v}_T}^2
      \\
      &\qquad
      + h_T^{-1}\norm{\bvec{L}^2(F;\Real^d)}{\vlproj{k}{F}(\PST\uvec{v}_T - \vlproj{k}{T}\PST\uvec{v}_T)}^2
      \\
      &\quad\lesssim
      h_T^{-2}\norm{\bvec{L}^2(T;\Real^d)}{\bvec{v}_T - \vlproj{k}{T}\PST\uvec{v}_T}^2
      + h_T^{-1}\norm{\bvec{L}^2(F;\Real^d)}{\bvec{v}_F - \vlproj{k}{F}\PST\uvec{v}_T}^2
      \\
      &\qquad
      + h_T^{-2}\norm{\bvec{L}^2(T;\Real^d)}{\PST\uvec{v}_T - \vlproj{k}{T}\PST\uvec{v}_T}^2
      \\
      &\quad\lesssim
      h_T^{-2}\norm{\bvec{U},T}{\uvec{v}_T - \IT\PST\uvec{v}_T}^2
      + \norm{\btens{L}^2(T;\Real^{d\times d})}{\GRAD\PST\uvec{v}_T}^2
      \\
      &\quad\lesssim
      \frac{\min(1,\Cf{T}^{-1})}{h_T^2}\norm{\bvec{U},T}{\uvec{v}_T - \IT\PST\uvec{v}_T}^2
      + \norm{\btens{L}^2(T;\Real^{d\times d})}{\GT\uvec{v}_T}^2
      = \norm{\stokes,T}{\uvec{v}_T}^2,
    \end{aligned}
  \end{equation}
  where we have used the $L^2$-boundedness of $\vlproj{k}{F}$ along with discrete trace inequalities in the second passage,
  the definition \eqref{eq:norm.U} of $\norm{\bvec{U},T}{{\cdot}}$ along with $\Cf{T}<1$ for the first two terms and the approximation properties of $\vlproj{k}{T}$ for the last term in the third passage,
  and concluded noticing that $1 = \min(1,\Cf{T}^{-1})$ and that $\GRAD\PST\uvec{v}_T$ is by definition the $L^2$-orthogonal projection of $\GT\uvec{v}_T$ on $\Goly{k}(T)^d$ (see \eqref{eq:PST}), so that $\norm{\btens{L}^2(T;\Real^{d\times d})}{\GRAD\PST\uvec{v}_T}\le\norm{\btens{L}^2(T;\Real^{d\times d})}{\GT\uvec{v}_T}$.
  Plugging \eqref{eq:est.vT-vF} into \eqref{eq:est.norm.grad.vT} and using the fact that $\card(\FT)\lesssim 1$ by mesh regularity, we get
  $\norm{\btens{L}^2(T;\Real^{d\times d})}{\GRAD\bvec{v}_T}^2\lesssim\norm{\stokes,T}{\uvec{v}_T}^2$,
  which is the sought estimate for the first term in the left-hand side of \eqref{eq:seminorm.1.T:lower.bound}.
  The fact second term is $\lesssim\norm{\stokes,T}{\uvec{v}_T}^2$ is an immediate consequence of \eqref{eq:est.vT-vF} along with $\card(\FT)\lesssim 1$.
\end{proof}

\begin{remark}[HHO stabilisation]
  It is not difficult to check that the bilinear form $\UT\times\UT\ni(\uvec{w}_T,\uvec{v}_T)\mapsto(\uvec{w}_T - \IT\PST\uvec{w}_T, \uvec{v}_T - \IT\PST\uvec{v}_T)_{\bvec{U},T}$ matches \cite[Assumption~8.10]{Di-Pietro.Droniou:20} if $\Cf{T}<1$.
  As a matter of fact, this bilinear form is clearly positive-semidefinite, it satisfies the requested seminorm equivalence by \eqref{eq:seminorm.1.T:upper.bound} and \eqref{eq:seminorm.1.T:lower.bound}, and is polynomially consistent since it only depends on its arguments through the difference operators defined by \cite[Eq.~(8.30)]{Di-Pietro.Droniou:20}.
\end{remark}

\subsection{Darcy term}

Let again $T\in\Th$.
The discretisation of the Darcy and coupling terms hinges on the \emph{discrete divergence} $\DT:\UT\to\Poly{k}(T)$ such that
\begin{equation}\label{eq:DT}
  \DT\uvec{v}_T \coloneq \tr(\GT\uvec{v}_T)\qquad\forall\uvec{v}_T\in\UT.
\end{equation}
Based on this operator, we define the \emph{Darcy potential} $\PDT:\UT\to\vPoly{k}(T;\Real^d)$ such that, for all $\uvec{v}_T\in\UT$ and all $(q,\bvec{w})\in\Poly{k+1}(T)\times\cGoly{k}(T)$,
\begin{equation}\label{eq:PDT}
  \int_T\PDT\uvec{v}_T\cdot(\GRAD q + \bvec{w})
  = -\int_T\DT\uvec{v}_T~q
  + \sum_{F\in\FT}\omega_{TF}\int_F(\bvec{v}_F\cdot\normal_F)~q
  + \int_T\bvec{v}_T\cdot\bvec{w}.
\end{equation}
This Darcy potential will play a key role in the discretisation of the source term, to ensure that the scheme is fully robust in the whole range of friction coefficients; see Remark \ref{rem:disc.source}.

\begin{remark}[Link with DDR]\label{rem:link.Xdivh}
  Recall the following Discrete De Rham $\Hdiv{\Omega}$-like space (in dimension $d=2$, this space corresponds to the two-dimensional DDR space for the curl rotated by $\nicefrac{\pi}{2}$):
  \begin{multline*}
  \Xdiv{h}\coloneq\Big\{
  \uvec{v}_h = \big(
  (\bvec{v}_T)_{T\in\Th}, (v_F)_{F\in\Fh}
  \big)\st
  \\
  \text{%
    $\bvec{v}_T\in\Goly{k-1}(T)\oplus\cGoly{k}(T)$
    for all $T\in\Th$ and
    $v_F\in\Poly{k}(F)$ for all $F\in\Fh$
  }
  \Big\}.
  \end{multline*}
  Noticing that $\Goly{k-1}(T)\oplus\cGoly{k}(T)\subset \vPoly{k}(T;\Real^d)$ (cf.\ \eqref{eq:Goly.cGoly}), this space naturally injects into $\Uh$ through the mapping $\Xdiv{h}\ni\uvec{v}_h\mapsto\big((\bvec{v}_T)_{T\in\Th},(v_F\normal_F)_{F\in\Fh}\big)\in\Uh$.
  It can be checked that the discrete divergence \eqref{eq:DT} and the Darcy potential \eqref{eq:PDT} only depend on the polynomial components shared by $\UT$ and $\Xdiv{T}$, and that, in the three-dimensional case, they coincide with the corresponding DDR operators respectively defined by \cite[Eqs.~(3.32) and~(4.9)--(4.10)]{Di-Pietro.Droniou:21*1} (in the two-dimensional case and accounting for the above-mentioned rotation of the spaces, $\DT$ corresponds to the two-dimensional DDR face curl and $\PDT$ to the tangential face reconstruction rotated by a right angle, see \cite[Eqs.~(3.19) and~(3.22)--(3.23)]{Di-Pietro.Droniou:21*1}).
\end{remark}

Accounting for the previous remark and recalling \cite[Eq.~(4.12) and (4.13)]{Di-Pietro.Droniou:21*1} for the three-dimensional case (\cite[Eqs.~(3.24) and (3.25)]{Di-Pietro.Droniou:21*1} for the two-dimensional case), it holds
\begin{alignat}{4}\label{eq:lproj.k-1.PDT}
  \vlproj{k-1}{T}\PDT\uvec{v}_T &= \vlproj{k-1}{T}\bvec{v}_T
  &\qquad&\forall\uvec{v}_T\in\UT,
  \\
  \label{eq:polynomial.consistency:PDT}
  \PDT\IT\bvec{v} &= \bvec{v}
  &\qquad&\forall\bvec{v}\in\vPoly{k}(T;\Real^d).
\end{alignat}
The approximation properties of $\PDT$ in the $L^2$-norm have been studied in \cite[Theorem~6]{Di-Pietro.Droniou:21*1}.
The following proposition extends the above results to general Hilbert seminorms.

\begin{proposition}[Approximation properties of the Darcy potential]
  Let an integer $r\in\{0,\ldots,k\}$ be given.
  Then, for all $T\in\Th$, all $\bvec{v}\in\bvec{H}^{r+1}(T;\Real^d)$, and all $m\in\{0,\ldots,r+1\}$,
  \begin{equation}\label{eq:approximation:PDT}
    \seminorm{\bvec{H}^m(T;\Real^d)}{\bvec{v} - \PDT\IT\bvec{v}}
    \lesssim h_T^{r+1-m}\seminorm{\bvec{H}^{r+1}(T;\Real^d)}{\bvec{v}}.
  \end{equation}
\end{proposition}

\begin{proof}  
  By \cite[Proposition~1.35]{Di-Pietro.Droniou:20}, $\PDT\IT:\bvec{H}^1(T;\Real^d)\to\vPoly{k}(T;\Real^d)$ is a projector owing to \eqref{eq:polynomial.consistency:PDT}.
  By \cite[Lemma~1.43]{Di-Pietro.Droniou:20}, it then suffices to prove that, for all $\bvec{v}\in\bvec{H}^1(T;\Real^d)$,
  \begin{alignat}{4}\label{eq:approximation:PDT:condition.m=0}
    \norm{\bvec{L}^2(T;\Real^d)}{\PDT\IT\bvec{v}}
    &\lesssim
    \norm{\bvec{L}^2(T;\Real^d)}{\bvec{v}}
    + h_T\seminorm{\bvec{H}^1(T;\Real^d)}{\bvec{v}}
    &\qquad&\text{if $m=0$},
    \\\label{eq:approximation:PDT:condition.m>=1}
    \seminorm{\bvec{H}^1(T;\Real^d)}{\PDT\IT\bvec{v}}
    &\lesssim\seminorm{\bvec{H}^1(T;\Real^d)}{\bvec{v}}
    &\qquad&\text{if $m\ge 1$}.
  \end{alignat}
  To prove \eqref{eq:approximation:PDT:condition.m=0}, it suffices to recall Remark~\ref{rem:link.Xdivh} and use \cite[Eqs.~(4.24) and~(4.28)]{Di-Pietro.Droniou:21*1} for the three-dimensional case (or \cite[Eqs.~(4.23)]{Di-Pietro.Droniou:21*1} for the two-dimensional case).
  To prove \eqref{eq:approximation:PDT:condition.m>=1}, we write
  \[
  \begin{alignedat}{4}
    \seminorm{\bvec{H}^1(T;\Real^d)}{\PDT\IT\bvec{v}}
    &= \seminorm{\bvec{H}^1(T;\Real^d)}{\PDT\IT(\bvec{v} - \vlproj{0}{T}\bvec{v})}
    \\
    &\lesssim
    h_T^{-1}\norm{\bvec{L}^2(T;\Real^d)}{\PDT\IT(\bvec{v} - \vlproj{0}{T}\bvec{v})}
    &\qquad&
    \text{\cite[Eq.~(1.46)]{Di-Pietro.Droniou:20}}
    \\
    &\lesssim
    h_T^{-1}\norm{\bvec{L}^2(T;\Real^d)}{\bvec{v} - \vlproj{0}{T}\bvec{v}}
    + \seminorm{\bvec{H}^1(T;\Real^d)}{\bvec{v} - \vlproj{0}{T}\bvec{v}}
    &\qquad&
    \text{Eq.~\eqref{eq:approximation:PDT:condition.m=0}}
    \\
    &\lesssim\seminorm{\bvec{H}^1(T;\Real^d)}{\bvec{v}},
  \end{alignedat}
  \]
  where the first line follows using the polynomial consistency \eqref{eq:polynomial.consistency:PDT} of $\PDT$ to write $0 = \seminorm{\bvec{H}^1(T;\Real^d)}{\vlproj{0}{T}\bvec{v}} = \seminorm{\bvec{H}^1(T;\Real^d)}{\PDT\IT\vlproj{0}{T}\bvec{v}}$,
  while the conclusion follows from a Poincar\'e--Wirtinger inequality on the zero-average function $\bvec{v} - \vlproj{0}{T}\bvec{v}$.
\end{proof}

Let $\tPDT:\UT\to\vPoly{k}(T;\Real^d)$ be such that
\begin{equation}\label{eq:tPDT}
  \tPDT\uvec{v}_T
  \coloneq\tv{\Cf{T}<1}\bvec{v}_T + \tv{\Cf{T}\ge 1}\PDT\uvec{v}_T
  \qquad\forall\uvec{v}_T\in\UT.
\end{equation}
The Darcy term in \eqref{eq:strong:momentum} is discretised by means of the bilinear form $a_{\nu,h}:\Uh\times\Uh\to\Real$ such that, for all $(\uvec{w}_h,\uvec{v}_h)\in\Uh\times\Uh$,
\begin{equation}\label{eq:a.nu}
  a_{\nu,h}(\uvec{w}_h,\uvec{v}_h)
  \coloneq\sum_{T\in\Th}\nu_Ta_{\darcy,T}(\uvec{w}_T,\uvec{v}_T)
\end{equation}
with, for all $T\in\Th$,
\begin{equation}\label{eq:a.darcy.T}
  a_{\darcy,T}(\uvec{w}_T,\uvec{v}_T)\coloneq
  \int_T\tPDT\uvec{w}_T\cdot\tPDT\uvec{v}_T + \min(1,\Cf{T})(
  \uvec{w}_T - \IT\PDT\uvec{w}_T,
  \uvec{v}_T - \IT\PDT\uvec{v}_T
  )_{\bvec{U},T}.
\end{equation}

We define the following induced norms: For all $\uvec{v}_h\in\Uh$,
\begin{equation}\label{eq:norm.nu.darcy}
  \norm{\nu,h}{\uvec{v}_h}\coloneq a_{\nu,h}(\uvec{v}_h,\uvec{v}_h)^{\nicefrac12}
  \quad\text{and\quad $\norm{\darcy,T}{\uvec{v}_T}\coloneq a_{\darcy,T}(\uvec{v}_T,\uvec{v}_T)^{\nicefrac12}$ for all $T\in\Th$.
  }
\end{equation}

\subsection{Coupling}

The coupling terms in \eqref{eq:strong:momentum} and \eqref{eq:strong:mass} are discretised by the bilinear form $b_h:\Uh\times\Poly{k}(\Th)\to\Real$ such that, for all $(\uvec{v}_h,q_h)\in\Uh\times\Poly{k}(\Th)$,
\begin{equation}\label{eq:bh}
  b_h(\uvec{v}_h,q_h)
  \coloneq-\sum_{T\in\Th}\int_T\DT\uvec{v}_T~q_T,
\end{equation}
where $q_T$ denotes the restriction of $q_h$ to $T$.
Recalling \cite[Eq.~(8.36)]{Di-Pietro.Droniou:20}, it holds:
For all $\bvec{v}\in\bvec{H}^1(\Omega;\Real^d)$,
\begin{equation}\label{eq:consistency:bh}
  b_h(\Ih\bvec{v},q_h) = -\int_\Omega\DIV\bvec{v}~q_h
  \qquad\forall q_h\in\Poly{k}(\Th).
\end{equation}

\subsection{Discrete problem and main results}

The discrete problem reads:
Find $(\uvec{u}_h,p_h)\in\UhD\times\Ph$ such that
\begin{equation}\label{eq:discrete}
  \begin{alignedat}{4}
    a_{\mu,h}(\uvec{u}_h,\uvec{v}_h)
    + a_{\nu,h}(\uvec{u}_h,\uvec{v}_h)
    + b_h(\uvec{v}_h,p_h)
    &= \sum_{T\in\Th}\int_T\bvec{f}\cdot\tPDT\uvec{v}_T
    &\qquad&\forall\uvec{v}_h\in\UhD,
    \\
    -b_h(\uvec{u}_h,q_h)
    &= \int_\Omega gq_h
    &\qquad&\forall q_h\in\Ph.
  \end{alignedat}
\end{equation}
The equivalent variational formulation is:
Find $(\uvec{u}_h,p_h)\in\UhD\times\Ph$ such that
\begin{equation}\label{eq:discrete:variational}
  \mathcal{A}_h((\uvec{u}_h,p_h),(\uvec{v}_h,q_h))
  = \sum_{T\in\Th}\int_T\bvec{f}\cdot\tPDT\uvec{v}_T
  + \int_\Omega gq_h,
\end{equation}
with $\mathcal{A}_h:\big(\Uh\times\Ph\big)^2\to\Real$ such that, for all $(\uvec{w}_h,r_h)$ and all $(\uvec{v}_h,q_h)$ in $\Uh\times\Ph$,
\begin{equation}\label{eq:Ah}
  \mathcal{A}_h((\uvec{w}_h,r_h),(\uvec{v}_h,q_h))
  \coloneq
  a_{\mu,h}(\uvec{w}_h,\uvec{v}_h)
  + a_{\nu,h}(\uvec{w}_h,\uvec{v}_h)
  + b_h(\uvec{v}_h,r_h)
  -b_h(\uvec{w}_h,q_h).
\end{equation}

Recalling \eqref{eq:norm.mu.stokes} and \eqref{eq:norm.nu.darcy}, we equip the space $\UhD$ with the following natural energy norm: For all $\uvec{v}_h\in\UhD$,
\begin{equation}\label{eq:norm.mu.nu.h}
  \norm{\mu,\nu,h}{\uvec{v}_h}\coloneq\left(
  \norm{\mu,h}{\uvec{v}_h}^2 + \norm{\nu,h}{\uvec{v}_h}^2
  \right)^{\nicefrac12}
\end{equation}
and, given a linear form $\ell_h:\UhD\to\Real$, we denote its dual norm by
\[
\norm{\mu,\nu,h,*}{\ell_h}\coloneq\sup_{\uvec{v}_h\in\UhD\setminus\{\uvec{0}\}}\frac{\ell_h(\uvec{v}_h)}{\norm{\mu,\nu,h}{\uvec{v}_h}}.
\]
The bilinear form $a_{\mu,h} + a_{\nu,h}$ is $\norm{\mu,\nu,h}{{\cdot}}$-coercive with unit coercivity constant.
The well-posedness of \eqref{eq:discrete} then classically follows from the theory of mixed methods (see, e.g., \cite[Lemma~A.11]{Di-Pietro.Droniou:20}) thanks to the inf-sup condition on $b_h$ stated in the following lemma.
\begin{lemma}[Inf-sup condition on $b_h$]\label{lem:inf-sup}
  Letting $\beta\coloneq\left(\overline{\mu} + \overline{\nu}\right)^{-\nicefrac12}$, it holds, for all $q_h\in\Ph$,
  \[
  \beta\norm{L^2(\Omega)}{q_h}
  \lesssim
  \sup_{\uvec{v}_h\in\UhD\setminus\{\uvec{0}\}}\frac{b_h(\uvec{v}_h,q_h)}{\norm{\mu,\nu,h}{\uvec{v}_h}}.
  \]
\end{lemma}

\begin{proof}
  See Section \ref{sec:analysis:stability}.
\end{proof}

Let $H^m(\Th)$ be spanned by square-integrable scalar-valued functions on $\Omega$ whose restriction to every mesh element $T\in\Th$ is in $H^m(T)$, and denote by $\bvec{H}^m(\Th;\Real^d)$ its vector-valued counterpart.
Thanks to the presence of cut-off factors, the following error estimate is robust across the entire range of (local) regimes.

\begin{theorem}[Error estimate]\label{thm:error.estimate}
  Denote by $(\bvec{u},p)\in\bvec{H}^1_0(\Omega;\Real^d)\times L^2_0(\Omega)$ the unique solution to the standard weak formulation of \eqref{eq:strong} and by $(\uvec{u}_h,p_h)\in\UhD\times\Ph$ the unique solution of the numerical scheme \eqref{eq:discrete} (or, equivalently, \eqref{eq:discrete:variational}).
  Then, recalling the notation \eqref{eq:tv} for the truth value of a logical proposition and assuming, for some $r\in\{0,\ldots,k\}$, $\bvec{u}\in\bvec{H}^{r+2}(\Th;\Real^d)$, $p\in H^1(\Omega)$, and, for all $T\in\Th$, $p\in H^{r+1+\tv{\Cf{T}\ge 1}}(T)$, it holds,
  \begin{equation}\label{eq:error.estimate}
    \begin{aligned}
      &
      \norm{\mu,\nu,h}{\uvec{u}_h - \Ih\bvec{u}}^2+\norm{L^2(\Omega)}{p_h - \lproj{k}{h} p}^2
      \\
      &\quad
      \lesssim\frac{1}{\gamma^2}\Bigg[
        \begin{aligned}[t]
          &\sum_{T\in\Th}\mu_T\min(1,\Cf{T}^{-1})h_T^{2(r+1)}\seminorm{\bvec{H}^{r+2}(T;\Real^d)}{\bvec{u}}^2
          + \sum_{T\in\Th}\nu_T\min(1,\Cf{T})h_T^{2(r+1)}\seminorm{\bvec{H}^{r+1}(T;\Real^d)}{\bvec{u}}^2
          \\
          &+ \sum_{T\in\Th}\left(
          \mu_T^{-1}\tv{\Cf{T}<1}h_T^{2(r+1)}\seminorm{H^{r+1}(T)}{p}^2
          + \nu_T^{-1}\tv{\Cf{T}\ge 1}h_T^{2(r+1)}\seminorm{H^{r+2}(T)}{p}^2    
          \right)
          \Bigg],
        \end{aligned}
    \end{aligned}
  \end{equation}
  where $\gamma^{-2}\coloneq 4\beta^{-4} + 8\beta^{-2} + 1$ with $\beta$ as in Lemma~\ref{lem:inf-sup}, while, for all $T\in\Th$, $\nu_T^{-1}\tv{\Cf{T}\ge 1}\coloneq 0$ if $\nu_T = 0$.
\end{theorem}

\begin{proof}
  See Section \ref{sec:analysis:convergence}.
\end{proof}
\begin{remark}[Robustness of the error estimate and application to the Darcy problem]\label{rem:pure.darcy}
  In the spirit of \cite[Remark~13]{Botti.Di-Pietro.ea:18}, the presence of the cutoff factors $\min(1,\Cf{T}^{-1})$, $\min(1,\Cf{T})$, $\mu_T^{-1}\tv{\Cf{T}<1}$, and $\nu_T^{-1}\tv{\Cf{T}\ge 1}$ makes the above estimate robust across the entire range $\Cf{T}\in\lbrack 0,+\infty)$.
  
  The case $\Cf{T} = +\infty$ corresponds to the pure Darcy problem, which is the singular limit obtained assuming $\min_\Omega\nu>0$ and $\Cf{T} = +\infty$ for all $T\in\Th$.
  In this case, a more in-depth discussion is in order.
  Denoting by $\gamma_{\normal}$ the normal trace operator on $\partial\Omega$, the space for the velocity becomes $\HZdiv{\Omega}\coloneq\{\bvec{v}\in\Hdiv{\Omega}\st\text{$\gamma_{\normal}(\bvec{v}) = 0$ on $\partial\Omega$}\}$, and the weak formulation of \eqref{eq:strong} yields the Darcy problem in mixed form.
  The error estimate \eqref{eq:error.estimate} remains valid under the regularity assumption $\bvec{u}\in\bvec{H}^{r+1}(\Th;\Real^d)$, and provided the following conventions are adopted:
  $\mu_T^{-1}\tv{\Cf{T}<1}\coloneq 0$ and, for any $\bvec{v}\in\HZdiv{\Omega}\cap\bvec{H}^1(\Th;\Real^d)$, all the components of the boundary values of $\Ih\bvec{v}$ are forced to zero, i.e., $(\Ih\bvec{v})_F \coloneq \bvec{0}$ for all $F\in\Fhb$.
  Notice that the tangential components of the velocity on boundary faces do not appear in the formulation of the method when $\mu = 0$.
    To check this fact:
  \begin{itemize}
  \item Concerning the Darcy contribution $a_{\darcy,T}$ (cf.\ \eqref{eq:a.darcy.T}), recall Remark~\ref{rem:link.Xdivh} for the consistent term while, for the stabilisation term, notice that, by \eqref{eq:prod.U}, boundary faces are not present in $(\cdot,\cdot)_{\bvec{U},T}$ since $\Cf{T}\ge 1$ for all $T\in\Th$ ;
  \item Concerning the coupling term $b_h$ (cf.\ \eqref{eq:bh}), notice that the following equivalent formulation results applying the definition \eqref{eq:GT} of $\GT$ with $\btens{\tau} = q_T\Id\coloneq (q_h)_{|T}\Id$ for all $T\in\Th$:
    \[
    b_h(\uvec{v}_h,q_h)
    = \sum_{T\in\Th}\left(
      \int_T\bvec{v}_T\cdot\GRAD q_T
      - \sum_{F\in\FT}\omega_{TF}\int_F(\bvec{v}_F\cdot\normal_F)~q_T
      \right),
      \]
      clearly showing that $b_h$ is independent of the tangential component of $\bvec{v}_F$ for all $F\in\Fh$.
  \end{itemize}
  
  The method obtained for the pure Darcy problem has more unknowns than, say, the mixed method of \cite{Di-Pietro.Ern:17} or a similar one that could be obtained starting from the space $\Xdiv{h}$ of \cite{Di-Pietro.Droniou:21*1}.
  In particular, the tangential components of interface unknowns are not present in the consistency term of $a_{\darcy,T}$ (see again Remark~\ref{rem:link.Xdivh}), but are controlled by the stabilisation term.
  Despite this difference in the discrete space for the flux, the estimate for the error on $\bvec{u}$ resulting from \eqref{eq:error.estimate} in the pure Darcy case is analogous to the one given in \cite[Theorem~6]{Di-Pietro.Ern:17} (where the highest regularity case corresponding to $r = k$ is considered).
\end{remark}


\section{Numerical tests}\label{sec:numerical.tests}

  In this section we numerically assess the convergence properties of the scheme \eqref{eq:discrete} for different values of the friction coefficient (including the limit cases) and on both standard and genuinely polyhedral meshes.

The code used for the numerical tests is part of the open source C++ \texttt{HArDCore3D} library; see \url{https://github.com/jdroniou/HArDCore}.
In order to reduce the size of the global linear systems, static condensation was applied the scheme \eqref{eq:discrete} in accordance with the principles outlined in \cite[Appendix B]{Di-Pietro.Droniou:20}; see \cite[Section 6]{Di-Pietro.Ern.ea:16*1} for a discussion specific to the Stokes equations and \cite{Botti.Di-Pietro:22} for a study of the effect of static condensation on $p$-multilevel preconditioners for the Stokes problem.
Specifically, we have chosen to locally eliminate all element degrees of freedom except for the average value of the pressure inside each element. The linear systems were solved using the \texttt{Intel MKL PARDISO} library (see \url{https://software.intel.com/en-us/mkl}). 

The parameter $\regT$ in \eqref{eq:prod.U} was chosen as $\frac{h_T^3}{|T|}\card(\FT)$, to give a larger weight to the element contribution in \eqref{eq:norm.U} when $T$ is elongated or has many faces: this compensates the relatively larger contribution, in these circumstances, of the boundary terms in this local norm. We have also applied scalings to the stabilisation terms (detailed in each section).
Introducing scalings in the stabilisation terms is not strictly necessary to observe the convergence of the scheme at the expected rates, but we noticed that they improve the magnitudes of the relative errors. Understanding the optimal scaling of stabilisations involved in polytopal methods is an ongoing subject of investigation; here, these numbers were found by quick trial and error on inexpensive tests (low degree $k$, coarse meshes), before being used in all the tests below.

\subsection{Convergence in various regimes}\label{sec:conv.various.regimes}

Following \cite{Botti.Di-Pietro.ea:18}, we consider a constant viscosity $\mu$ and inverse permeability $\nu$, and we evaluate the relative velocity--pressure error
\[
E_{\bvec{u},p}=\frac{\left(\norm{\mu,\nu,h}{\uvec{u}_h-\Ih\bvec{u}}^2+\norm{L^2(\Omega)}{p_h-\lproj{k}{h}p}^2\right)^{\nicefrac12}}{\left(\norm{\mu,\nu,h}{\Ih\bvec{u}}^2+\norm{L^2(\Omega)}{\lproj{k}{h}p}^2\right)^{\nicefrac12}}
\]
when the nature of the exact solution $(\bvec{u},p)$ is determined by the global friction coefficient $\Cf{\Omega}=\nicefrac{\nu}{\mu}$, with the convention $\Cf{\Omega}=+\infty$ if $\mu=0$. Specifically, we consider the domain $\Omega=(0,1)^3$ and, setting $\chi_\stokes(\Cf{\Omega}) \coloneq \exp(-\Cf{\Omega})$, the pressure and velocity are chosen as
\begin{align}
p(x,y,z)={}&\sin(2\pi x)\sin(2\pi y)\sin(2\pi z)\quad\forall (x,y,z)\in\Omega,\label{eq:trigo.p}\\
\bvec{u}={}&\chi_\stokes(\Cf{\Omega})\bvec{u}_\stokes+(1-\chi_\stokes(\Cf{\Omega}))\bvec{u}_\darcy,\nonumber
\end{align}
where $\bvec{u}_\stokes$ and $\bvec{u}_\darcy$ are the velocities obtained in the Stokes ($\Cf{\Omega}=0$) and Darcy ($\Cf{\Omega}=+\infty$) limits, and are given by
\[
\begin{gathered}
  \bvec{u}_\stokes(x,y,z)
  =\frac12\begin{bmatrix}
  \sin(2\pi x)\cos(2\pi y)\cos(2\pi z)\\
  \cos(2\pi x)\sin(2\pi y)\cos(2\pi z)\\
  -2\cos(2\pi x)\cos(2\pi y)\sin(2\pi z)
  \end{bmatrix}\quad\forall (x,y,z)\in\Omega,\\
  \bvec{u}_\darcy=\begin{cases}
  -\nu^{-1}\GRAD p & \text{if $\nu>0$} ,\\
  \bvec{0} & \text{otherwise}.
  \end{cases}
\end{gathered}
\]
We notice that $\DIV \bvec{u}_\stokes=0$ and that $\nu\bvec{u}_\darcy+\GRAD p=0$; these are expected relations, respectively, for a solution of the incompressible Stokes equation, and for a solution of the Darcy equation in mixed form (when gravity is neglected).
  The meshes used for the test correspond to the families of Voronoi meshes ``Voro-small-0'', of tetrahedral meshes ``Tetgen-Cube-0'', and of random hexahedral meshes ``Random-Hexahedra'' available on the \texttt{HArDCore3D} repository. The stabilisation term in the Stokes contribution \eqref{eq:a.stokes.T} has been scaled by 3, and the stabilisation term in the Darcy contribution \eqref{eq:a.darcy.T} by 0.3.
The errors as functions of $h$ are presented in Figures \ref{fig:conv.voro}, \ref{fig:conv.tets} and \ref{fig:conv.hexa}, showing that the predicted convergence is observed in practice for all the considered mesh families and polynomial degrees, and that both the orders of convergence and the magnitudes of errors are robust in all regimes.
\begin{figure}[h!]\centering
  \ref{br.conv.voro}
  \vspace{0.50cm}\\
  \begin{minipage}{0.45\textwidth}
    \begin{tikzpicture}[scale=0.85]
      \begin{loglogaxis} [legend columns=4, legend to name=br.conv.voro]  
        \logLogSlopeTriangle{0.90}{0.4}{0.1}{1}{black};
        \addplot [mark=star, red] table[x=MeshSize,y=EnergyError] {dat/convergence/Voro-small-0_k0_sv1_sp1/data_rates.dat};
        \addlegendentry{$k=0$;}
        \logLogSlopeTriangle{0.90}{0.4}{0.1}{2}{black};
        \addplot [mark=*, blue] table[x=MeshSize,y=EnergyError] {dat/convergence/Voro-small-0_k1_sv1_sp1/data_rates.dat};
        \addlegendentry{$k=1$;}
        \logLogSlopeTriangle{0.90}{0.4}{0.1}{3}{black};
        \addplot [mark=o, olive] table[x=MeshSize,y=EnergyError] {dat/convergence/Voro-small-0_k2_sv1_sp1/data_rates.dat};
        \addlegendentry{$k=2$}
        \logLogSlopeTriangle{0.90}{0.4}{0.1}{4}{black};
        \addplot [mark=o, black] table[x=MeshSize,y=EnergyError] {dat/convergence/Voro-small-0_k3_sv1_sp1/data_rates.dat};
        \addlegendentry{$k=3$}
      \end{loglogaxis}            
    \end{tikzpicture}
    \subcaption{$\mu=\nu=1$}
  \end{minipage}
  \begin{minipage}{0.45\textwidth}
    \begin{tikzpicture}[scale=0.85]
      \begin{loglogaxis} 
        \logLogSlopeTriangle{0.90}{0.4}{0.1}{1}{black};
        \addplot [mark=star, red] table[x=MeshSize,y=EnergyError] {dat/convergence/Voro-small-0_k0_sv1_sp0/data_rates.dat};
        \logLogSlopeTriangle{0.90}{0.4}{0.1}{2}{black};
        \addplot [mark=*, blue] table[x=MeshSize,y=EnergyError] {dat/convergence/Voro-small-0_k1_sv1_sp0/data_rates.dat};
        \logLogSlopeTriangle{0.90}{0.4}{0.1}{3}{black};
        \addplot [mark=o, olive] table[x=MeshSize,y=EnergyError] {dat/convergence/Voro-small-0_k2_sv1_sp0/data_rates.dat};
        \logLogSlopeTriangle{0.90}{0.4}{0.1}{4}{black};
        \addplot [mark=o, black] table[x=MeshSize,y=EnergyError] {dat/convergence/Voro-small-0_k3_sv1_sp0/data_rates.dat};
      \end{loglogaxis}            
    \end{tikzpicture}
    \subcaption{$\mu=1$, $\nu=0$}
  \end{minipage}
  \\[0.5em]
  \begin{minipage}{0.45\textwidth}
    \begin{tikzpicture}[scale=0.85]
      \begin{loglogaxis} 
        \logLogSlopeTriangle{0.90}{0.4}{0.1}{1}{black};
        \addplot [mark=star, red] table[x=MeshSize,y=EnergyError] {dat/convergence/Voro-small-0_k0_sv0_sp1/data_rates.dat};
        \logLogSlopeTriangle{0.90}{0.4}{0.1}{2}{black};
        \addplot [mark=*, blue] table[x=MeshSize,y=EnergyError] {dat/convergence/Voro-small-0_k1_sv0_sp1/data_rates.dat};
        \logLogSlopeTriangle{0.90}{0.4}{0.1}{3}{black};
        \addplot [mark=o, olive] table[x=MeshSize,y=EnergyError] {dat/convergence/Voro-small-0_k2_sv0_sp1/data_rates.dat};
        \logLogSlopeTriangle{0.90}{0.4}{0.1}{4}{black};
        \addplot [mark=o, black] table[x=MeshSize,y=EnergyError] {dat/convergence/Voro-small-0_k3_sv0_sp1/data_rates.dat};
      \end{loglogaxis}            
    \end{tikzpicture}
    \subcaption{$\mu=0$, $\nu=1$}
  \end{minipage}
  \caption{Tests of Section \ref{sec:conv.various.regimes}, Voronoi meshes: errors $E_{u,p}$ with respect to the mesh size $h$}
  \label{fig:conv.voro}
\end{figure}
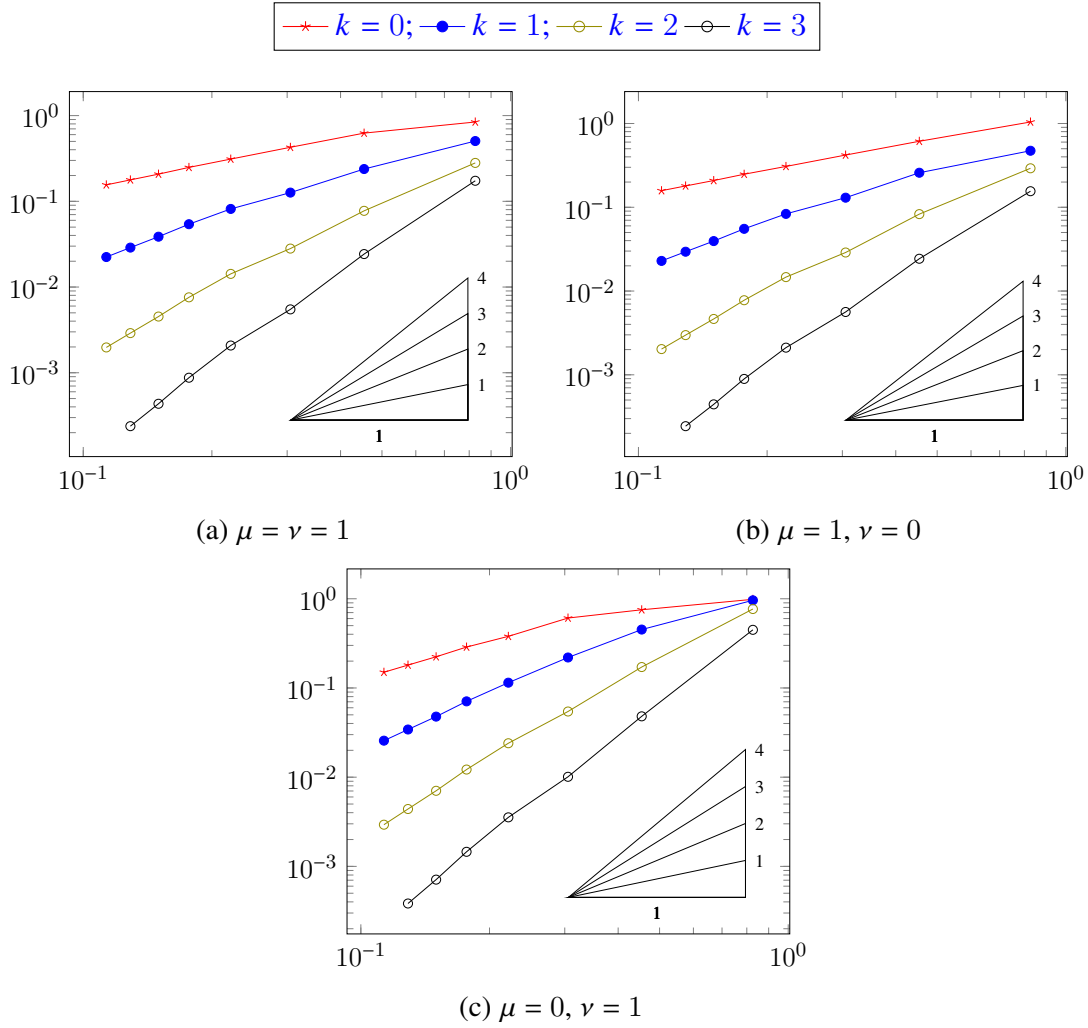

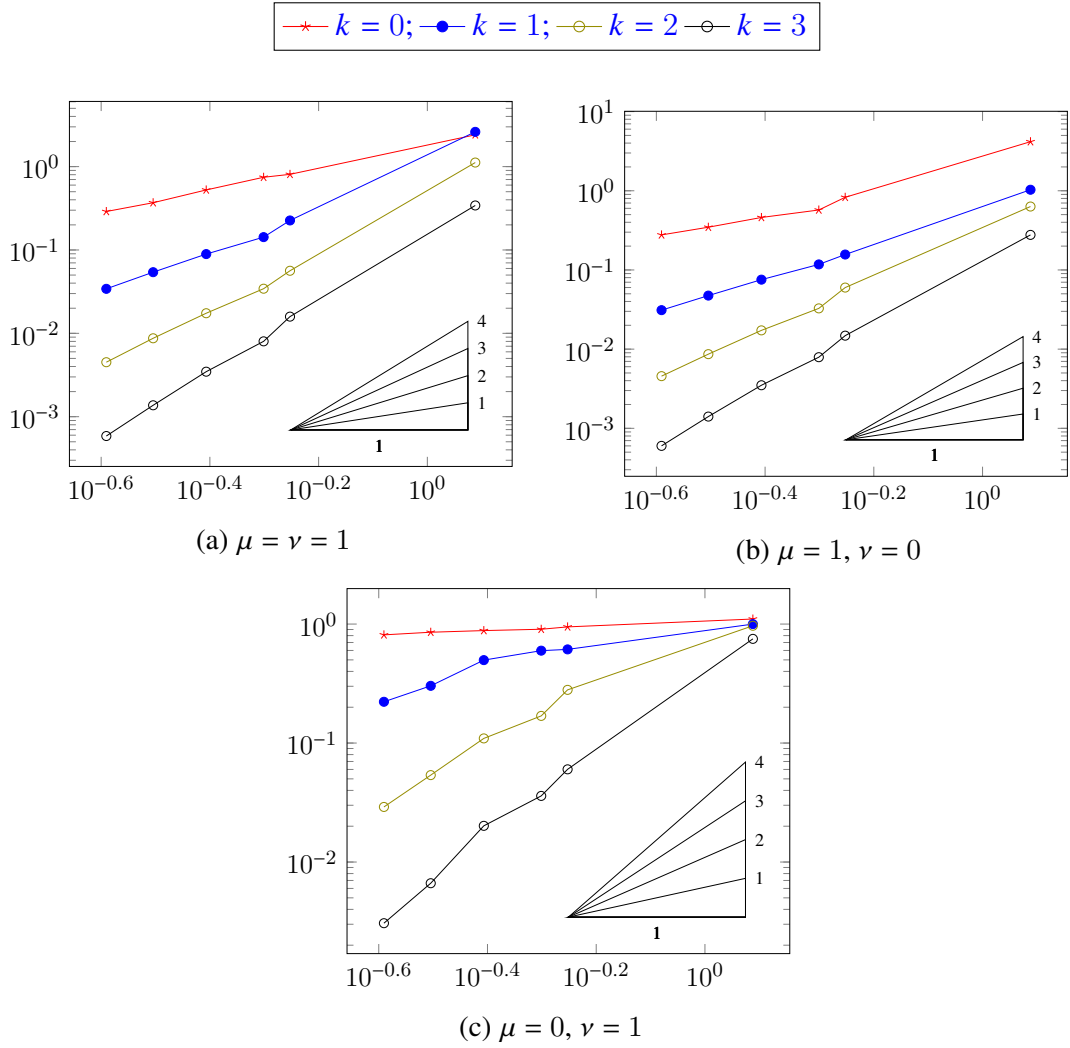
\begin{figure}[!h]\centering
  \ref{br.conv.tets}
  \vspace{0.50cm}\\
  \begin{minipage}{0.45\textwidth}
    \begin{tikzpicture}[scale=0.85]
      \begin{loglogaxis} [legend columns=4, legend to name=br.conv.tets]  
        \logLogSlopeTriangle{0.90}{0.4}{0.1}{1}{black};
        \addplot [mark=star, red] table[x=MeshSize,y=EnergyError] {dat/convergence/Tetgen-Cube-0_k0_sv1_sp1/data_rates.dat};
        \addlegendentry{$k=0$;}
        \logLogSlopeTriangle{0.90}{0.4}{0.1}{2}{black};
        \addplot [mark=*, blue] table[x=MeshSize,y=EnergyError] {dat/convergence/Tetgen-Cube-0_k1_sv1_sp1/data_rates.dat};
        \addlegendentry{$k=1$;}
        \logLogSlopeTriangle{0.90}{0.4}{0.1}{3}{black};
        \addplot [mark=o, olive] table[x=MeshSize,y=EnergyError] {dat/convergence/Tetgen-Cube-0_k2_sv1_sp1/data_rates.dat};
        \addlegendentry{$k=2$}
        \logLogSlopeTriangle{0.90}{0.4}{0.1}{4}{black};
        \addplot [mark=o, black] table[x=MeshSize,y=EnergyError] {dat/convergence/Tetgen-Cube-0_k3_sv1_sp1/data_rates.dat};
        \addlegendentry{$k=3$}
      \end{loglogaxis}            
    \end{tikzpicture}
    \subcaption{$\mu=\nu=1$}
  \end{minipage}
  \begin{minipage}{0.45\textwidth}
    \begin{tikzpicture}[scale=0.85]
      \begin{loglogaxis}
        \logLogSlopeTriangle{0.90}{0.4}{0.1}{1}{black};
        \addplot [mark=star, red] table[x=MeshSize,y=EnergyError] {dat/convergence/Tetgen-Cube-0_k0_sv1_sp0/data_rates.dat};
        \logLogSlopeTriangle{0.90}{0.4}{0.1}{2}{black};
        \addplot [mark=*, blue] table[x=MeshSize,y=EnergyError] {dat/convergence/Tetgen-Cube-0_k1_sv1_sp0/data_rates.dat};
        \logLogSlopeTriangle{0.90}{0.4}{0.1}{3}{black};
        \addplot [mark=o, olive] table[x=MeshSize,y=EnergyError] {dat/convergence/Tetgen-Cube-0_k2_sv1_sp0/data_rates.dat};
        \logLogSlopeTriangle{0.90}{0.4}{0.1}{4}{black};
        \addplot [mark=o, black] table[x=MeshSize,y=EnergyError] {dat/convergence/Tetgen-Cube-0_k3_sv1_sp0/data_rates.dat};
      \end{loglogaxis}            
    \end{tikzpicture}
    \subcaption{$\mu=1$, $\nu=0$}
  \end{minipage}
  \\[0.5em]
  \begin{minipage}{0.45\textwidth}
    \begin{tikzpicture}[scale=0.85]
      \begin{loglogaxis}
        \logLogSlopeTriangle{0.90}{0.4}{0.1}{1}{black};
        \addplot [mark=star, red] table[x=MeshSize,y=EnergyError] {dat/convergence/Tetgen-Cube-0_k0_sv0_sp1/data_rates.dat};
        \logLogSlopeTriangle{0.90}{0.4}{0.1}{2}{black};
        \addplot [mark=*, blue] table[x=MeshSize,y=EnergyError] {dat/convergence/Tetgen-Cube-0_k1_sv0_sp1/data_rates.dat};
        \logLogSlopeTriangle{0.90}{0.4}{0.1}{3}{black};
        \addplot [mark=o, olive] table[x=MeshSize,y=EnergyError] {dat/convergence/Tetgen-Cube-0_k2_sv0_sp1/data_rates.dat};
        \logLogSlopeTriangle{0.90}{0.4}{0.1}{4}{black};
        \addplot [mark=o, black] table[x=MeshSize,y=EnergyError] {dat/convergence/Tetgen-Cube-0_k3_sv0_sp1/data_rates.dat};
      \end{loglogaxis}            
    \end{tikzpicture}
    \subcaption{$\mu=0$, $\nu=1$}
  \end{minipage}
  \caption{Tests of Section \ref{sec:conv.various.regimes}, tetrahedral meshes: errors $E_{u,p}$ with respect to the mesh size $h$}
  \label{fig:conv.tets}
\end{figure}

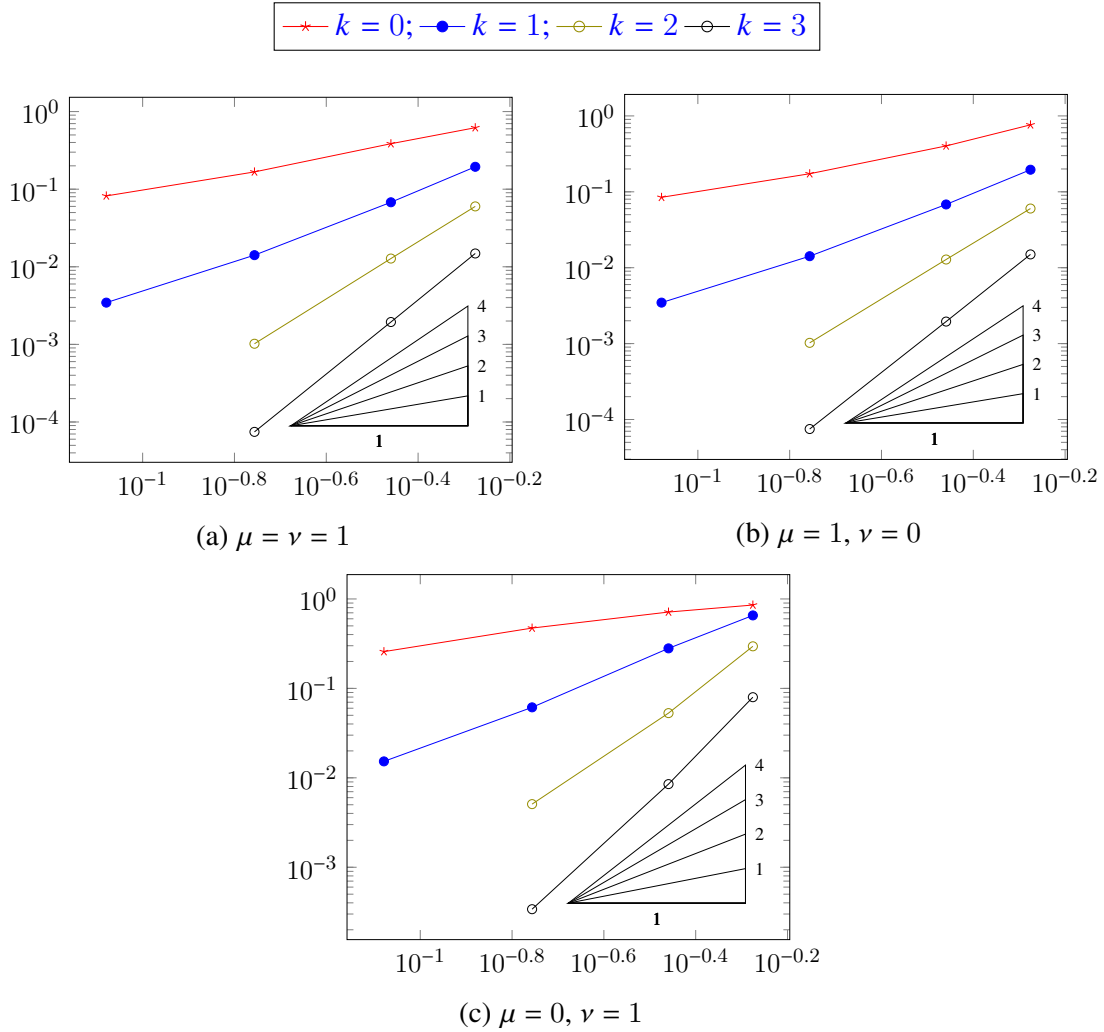
\begin{figure}[!h]\centering
  \ref{br.conv.hexa}
  \vspace{0.50cm}\\
  \begin{minipage}{0.45\textwidth}
    \begin{tikzpicture}[scale=0.85]
      \begin{loglogaxis} [legend columns=4, legend to name=br.conv.hexa]  
        \logLogSlopeTriangle{0.90}{0.4}{0.1}{1}{black};
        \addplot [mark=star, red] table[x=MeshSize,y=EnergyError] {dat/convergence/Random-Hexahedra_k0_sv1_sp1/data_rates.dat};
        \addlegendentry{$k=0$;}
        \logLogSlopeTriangle{0.90}{0.4}{0.1}{2}{black};
        \addplot [mark=*, blue] table[x=MeshSize,y=EnergyError] {dat/convergence/Random-Hexahedra_k1_sv1_sp1/data_rates.dat};
        \addlegendentry{$k=1$;}
        \logLogSlopeTriangle{0.90}{0.4}{0.1}{3}{black};
        \addplot [mark=o, olive] table[x=MeshSize,y=EnergyError] {dat/convergence/Random-Hexahedra_k2_sv1_sp1/data_rates.dat};
        \addlegendentry{$k=2$}
        \logLogSlopeTriangle{0.90}{0.4}{0.1}{4}{black};
        \addplot [mark=o, black] table[x=MeshSize,y=EnergyError] {dat/convergence/Random-Hexahedra_k3_sv1_sp1/data_rates.dat};
        \addlegendentry{$k=3$}
      \end{loglogaxis}            
    \end{tikzpicture}
    \subcaption{$\mu=\nu=1$}
  \end{minipage}
  \begin{minipage}{0.45\textwidth}
    \begin{tikzpicture}[scale=0.85]
      \begin{loglogaxis}
        \logLogSlopeTriangle{0.90}{0.4}{0.1}{1}{black};
        \addplot [mark=star, red] table[x=MeshSize,y=EnergyError] {dat/convergence/Random-Hexahedra_k0_sv1_sp0/data_rates.dat};
        \logLogSlopeTriangle{0.90}{0.4}{0.1}{2}{black};
        \addplot [mark=*, blue] table[x=MeshSize,y=EnergyError] {dat/convergence/Random-Hexahedra_k1_sv1_sp0/data_rates.dat};
        \logLogSlopeTriangle{0.90}{0.4}{0.1}{3}{black};
        \addplot [mark=o, olive] table[x=MeshSize,y=EnergyError] {dat/convergence/Random-Hexahedra_k2_sv1_sp0/data_rates.dat};
        \logLogSlopeTriangle{0.90}{0.4}{0.1}{4}{black};
        \addplot [mark=o, black] table[x=MeshSize,y=EnergyError] {dat/convergence/Random-Hexahedra_k3_sv1_sp0/data_rates.dat};
      \end{loglogaxis}            
    \end{tikzpicture}
    \subcaption{$\mu=1$, $\nu=0$}
  \end{minipage}
  \\[0.5em]
  \begin{minipage}{0.45\textwidth}
    \begin{tikzpicture}[scale=0.85]
      \begin{loglogaxis}
        \logLogSlopeTriangle{0.90}{0.4}{0.1}{1}{black};
        \addplot [mark=star, red] table[x=MeshSize,y=EnergyError] {dat/convergence/Random-Hexahedra_k0_sv0_sp1/data_rates.dat};
        \logLogSlopeTriangle{0.90}{0.4}{0.1}{2}{black};
        \addplot [mark=*, blue] table[x=MeshSize,y=EnergyError] {dat/convergence/Random-Hexahedra_k1_sv0_sp1/data_rates.dat};
        \logLogSlopeTriangle{0.90}{0.4}{0.1}{3}{black};
        \addplot [mark=o, olive] table[x=MeshSize,y=EnergyError] {dat/convergence/Random-Hexahedra_k2_sv0_sp1/data_rates.dat};
        \logLogSlopeTriangle{0.90}{0.4}{0.1}{4}{black};
        \addplot [mark=o, black] table[x=MeshSize,y=EnergyError] {dat/convergence/Random-Hexahedra_k3_sv0_sp1/data_rates.dat};
      \end{loglogaxis}            
    \end{tikzpicture}
    \subcaption{$\mu=0$, $\nu=1$}
  \end{minipage}
  \caption{Tests of Section \ref{sec:conv.various.regimes}, random hexahedral meshes: errors $E_{u,p}$ with respect to $h$}
  \label{fig:conv.hexa}
\end{figure}

Table \ref{tab:condnum} presents the condition numbers for $k=1$, on one member of each mesh family (the second for the Voronoi and tetrahedral meshes, the first for the random hexahedral meshes). These numbers show that the conditioning of the scheme is also robust in the limits $\mu\to 0$ or $\nu\to 0$; actually, this conditioning appears to be mostly driven by the strength of the Stokes terms: the magnitudes of the condition numbers are comparable in the balanced Stokes--Darcy and the pure Stokes regime, and much lower in the pure Darcy regime.
This observation is consistent with the fact that, based on standard estimates for pure diffusion problems (see, e.g., \cite{Ern.Guermond:06} for finite elements and \cite{Badia.Droniou.ea:22} for the HHO scheme), the conditioning for the Darcy problem is expected to scale with $h^{-1}$, as opposed to $h^{-2}$ for the Stokes problem.

\begin{table}\centering
  \begin{tabular}{c|c|c|c}
    \toprule
    Mesh & Voronoi & Tetrahedral & Random Hex.\\
    Num. of elements &  125 & 216 & 176 \\
    \midrule
    $\mu=\nu=1$ (balanced) & $1.4\times 10^4$ & $1.5\times 10^{5}$ & $1.1\times 10^4$\\
    $\mu=1$, $\nu=0$ (pure Stokes) & $1.3\times 10^{4}$ & $1.5\times 10^5$ & $1.1\times 10^4$\\
    $\mu=0$, $\nu=1$ (pure Darcy) & $190$ & $1.6\times 10^3$ & $412$\\
    \bottomrule
  \end{tabular}
  \caption{Condition numbers of the matrices in the tests of Section \ref{sec:conv.various.regimes}, for $k=1$ and one member of each mesh family.\label{tab:condnum}}
\end{table}

\begin{remark}[Numerical handling of $\Cf{T}=+\infty$]
  In the numerical implementation, the case $\Cf{T}=+\infty$ is handled using a threshold. Specifically, we fix $\epsilon=10^{-14}$ and use the following numerical value for $\Cf{T}$:
  \[
  \hCf{T}=\begin{cases}
  \max\left(\epsilon; \frac{\nu_T h_T^2}{\mu_T}\right) & \text{if $\mu_T>\epsilon$},
  \\
  \epsilon^{-1} & \text{if $\mu_T\le \epsilon$}.
  \end{cases}
  \]
  This choice ensures that $\hCf{T}$ is always well-defined (no division by $0$ occurs), and remains larger than $\epsilon$. This second constraint is required because the scheme uses $\hCf{T}^{-1}$, which needs to remain computable and not lead to a division by 0. Using $\hCf{T}$ ensures that we have a computable term in all circumstances, without having to do specific tests each time $\Cf{T}$ or $\Cf{T}^{-1}$ is required. 
  
  We note that the scheme only depends on the numerical friction coefficients through $\min(1,\hCf{T})$, $\min(1,\hCf{T}^{-1})$, $\tv{\hCf{T}\ge 1}$ and $\tv{\hCf{T}<1}$. The truth values computed with $\Cf{T}$ or $\hCf{T}$ are strictly the same provided that $\epsilon<1$. The substitution of $\hCf{T}$ for $\Cf{T}$ in the minima terms (which leads, for example, to using $\min(1,\hCf{T})=\epsilon=10^{-14}$ instead of 0 when $\nu_T=0$) has almost no impact on the computed numerical solution.
\end{remark}

\subsection{Convergence for discontinuous viscosity and permeability}\label{sec:conv.disc}

In this section, we still present convergence results towards a manufactured analytical solution, but in a more challenging setting than in the previous section: the viscosity and permeability are discontinuous in the domain, and the regime degenerates to a full Darcy limit in part of the domain. Specifically, we consider $\Omega=(0,1)^3$ split into a Stokes-dominated subdomain $\Omega_{\stokes}=(0,\nicefrac12)\times(0,1)^2$ and a pure Darcy subdomain $\Omega_{\darcy}=(\nicefrac12,1)\times(0,1)^2$. The exact pressure is still chosen as \eqref{eq:trigo.p}, while the exact velocity is $\bvec{u}=\bvec{u}_0+\chi_{\stokes}\bvec{u}_{\stokes} + \chi_{\darcy}\bvec{u}_{\darcy}$, where $\chi_i$ is the characteristic function of $\Omega_i$ and, for all $(x,y,z)\in\Omega$,
\[
\begin{gathered}
  \bvec{u}_0(x,y,z)=\begin{bmatrix}
  \exp(-y-z) \\ \sin(\pi y)\sin(\pi z) \\ yz
  \end{bmatrix},\quad
  \bvec{u}_{\stokes}(x,y,z)=\cos(\pi x)(x-0.5)\begin{bmatrix}
    y+z \\ y +\cos(\pi z) \\ \sin(\pi y)
  \end{bmatrix},
  \\
  \bvec{u}_{\darcy}(x,y,z)=\cos(\pi x)(x-0.5)\begin{bmatrix}
    \sin(\pi y)\sin(\pi z) \\ z^3 \\ y^2z^2
  \end{bmatrix}.
\end{gathered}
\]
The medium parameters are $(\mu,\nu)=(1,10^7)$ in $\Omega_{\stokes}$, and $(\mu,\nu)=(0,10^2)$ in $\Omega_{\darcy}$. The simulations are run on a family of Cartesian meshes (compatible with the interface $x=\nicefrac12$) respectively made of $2^3,4^3,8^3,16^3$ and $32^3$ cubes.

As $\bvec{u}_0$ does not depend on $x$, and thanks to the presence of the term $\cos(\pi x)(x-0.5)$, $\bvec{u}$ is continuous across the interface $x=\nicefrac12$ and $\GRAD \bvec{u}\normal=\bvec{0}$ on that interface; hence, $\bvec{u}\in \bvec{H}^1(\Omega;\Real^d)\cap \bvec{H}^\infty(\Th;\Real^d)$ and the source term $\bvec{f}=-\VDIV(\mu\GRAD u) + \nu\bvec{u}+\GRAD p$ belongs to $\bvec{L}^2(\Omega;\Real^d)$ (no singularity appears at the interface $x=\nicefrac12$). 

The Stokes stabilisation in \eqref{eq:a.stokes.T} has not been scaled (as scaling did not show in this case a significant impact on the magnitudes of the errors), but we have applied a scaling of $10^{-(k+1)}$ to the Darcy stabilisation in \eqref{eq:a.darcy.T}.
In Figure~\ref{fig:conv.disc}, we display the errors $E_{u,p}$ for polynomial degrees $k\in\{0,1,2\}$ as a function of the meshsize.
We notice a super-convergence effect linked to the use of Cartesian meshes: the energy error decays as $h^{k+2}$ instead of $h^{k+1}$ (the rate is even closer to $h^6$ than $h^5$ for $k=3$). However, on coarse meshes and/or for a small degree $k$, the magnitude of the error is quite large, polluted by a bad approximation of the pressure.

\begin{figure}[h!]\centering
   \ref{br.conv.disc}
   \vspace{0.50cm}\\  
  \begin{tikzpicture}[scale=0.85]
      \begin{loglogaxis}[legend columns=4, legend to name=br.conv.disc]  
        \logLogSlopeTriangle{0.90}{0.4}{0.1}{1}{black};
        \addplot [mark=star, red] table[x=MeshSize,y=EnergyError] {dat/convergence-discontinuous/Cubic-Cells_k0_sv1_sp1/data_rates.dat};
        \addlegendentry{$k=0$;}
        \logLogSlopeTriangle{0.90}{0.4}{0.1}{2}{black};
        \addplot [mark=*, blue] table[x=MeshSize,y=EnergyError] {dat/convergence-discontinuous/Cubic-Cells_k1_sv1_sp1/data_rates.dat};
        \addlegendentry{$k=1$;}
        \logLogSlopeTriangle{0.90}{0.4}{0.1}{3}{black};
        \addplot [mark=o, olive] table[x=MeshSize,y=EnergyError] {dat/convergence-discontinuous/Cubic-Cells_k2_sv1_sp1/data_rates.dat};
        \addlegendentry{$k=2$}
        \logLogSlopeTriangle{0.90}{0.4}{0.1}{4}{black};
        \addplot [mark=o, black] table[x=MeshSize,y=EnergyError] {dat/convergence-discontinuous/Cubic-Cells_k3_sv1_sp1/data_rates.dat};
        \addlegendentry{$k=3$}
        \logLogSlopeTriangle{0.90}{0.4}{0.1}{5}{black};
      \end{loglogaxis}            
    \end{tikzpicture}
  \caption{Tests of Section \ref{sec:conv.disc} (discontinuous viscosity and permeability): errors $E_{u,p}$ with respect to the mesh size $h$}
  \label{fig:conv.disc}
\end{figure}
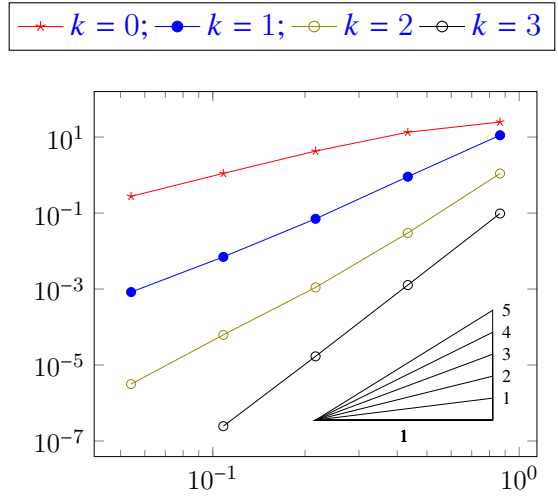

\subsection{Lid-driven cavity in porous medium}\label{sec:lid.cavity}

The tests in this section are inspired by situations described in \cite{Bernardi.Hecht.ea:10,Alvarez.Gatica.ea:16}. In these references, a V-crack is realised at the top of a homogeneous porous medium, and plays the role of a lid-driven cavity (with a Stokes-dominated model in this cavity, while the rest of the medium is modelled using pure Darcy flow), and low-order mixed finite elements on triangles/tetrahedra are used to simulate the flow. 

We consider here a cavity, where a pure Stokes flow occurs with viscosity $10^{-2}$, sitting in a porous medium with pure Darcy flow; the porous medium is heterogeneous, with permeability equal to $10^{-7}$ in the surrounding ``box'' and $10^{-2}$ in a ``wedge'' at the outset of the cavity; see Figure \ref{fig:cavity-domain}, left. The domain is $\Omega=(-1,2)\times (-1,2)\times(-2,0)$, with the cavity being $(0,1)^3$ and the wedge $\left\{(x,y,z)\in\Real^3\,:\,1<x<2\,,\;0<y<1\,,\;0.25(x-1)-0.75<z<0\right\}$. The domain has been meshed using \texttt{gmsh} (\href{https://gmsh.info/}{https://gmsh.info/}), with cubic elements in the cavity, and mostly tetrahedral elements in the porous medium (together with a few pyramidal elements around the interface between these two regions); see Figure \ref{fig:cavity-domain}, right, for an example of mesh, and Table \ref{tab:meshes} for more insight into the features of the mesh family. Notice, in particular, that the meshes are not quasi-uniform, and include small tetrahedra and pyramids, mostly located at the interfaces between the various subregions in the domain. The files describing the geometry are available in the \texttt{HArDCore} repository.
The scalings of the stabilisation terms are 3 and 0.3, as in Section \ref{sec:conv.various.regimes}.

\begin{figure}\centering
  \begin{minipage}{0.45\textwidth}
    \includegraphics[width=\textwidth]{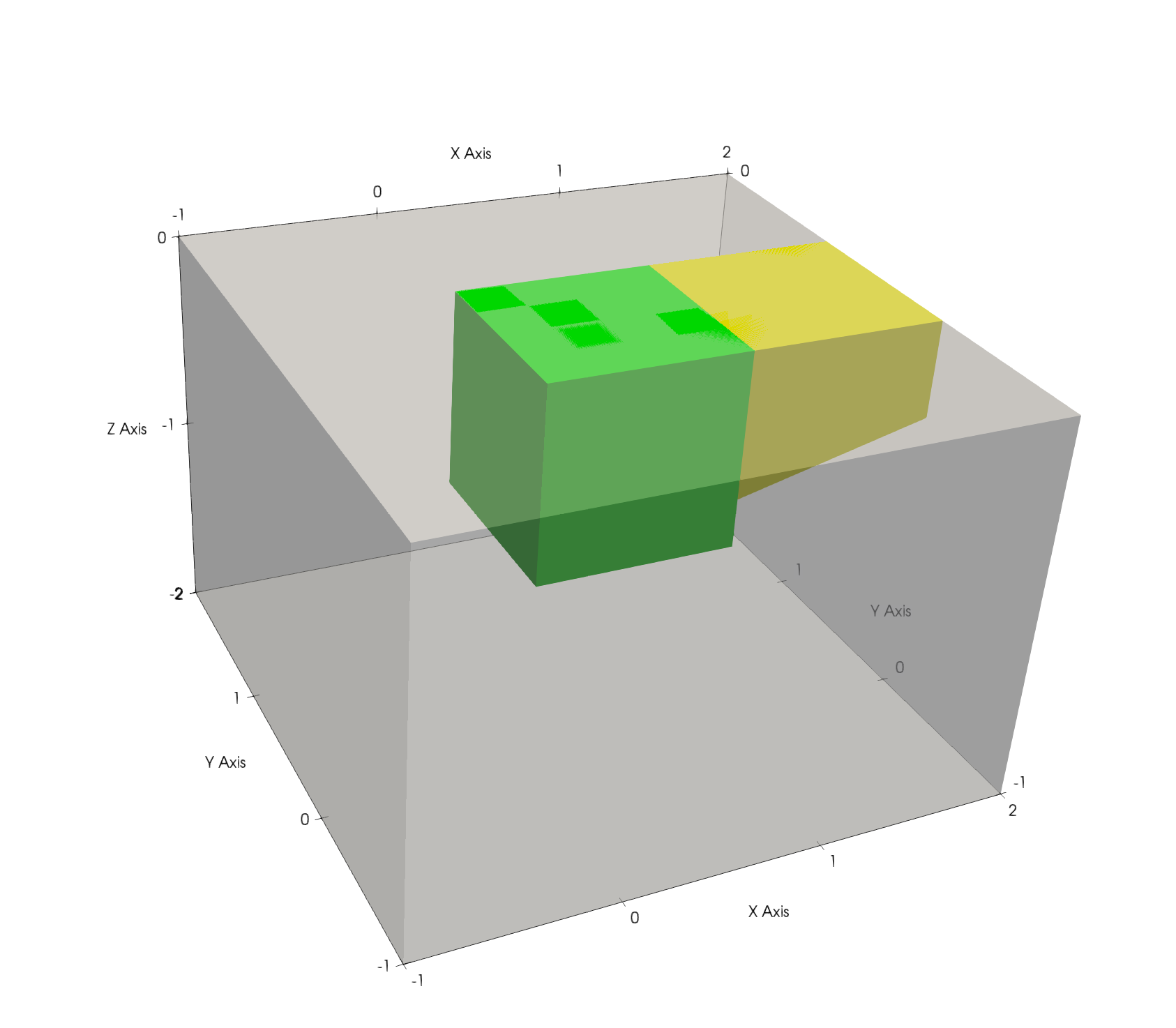}
  \end{minipage}
  \hspace{0.25cm}
  \begin{minipage}{0.45\textwidth}
    \includegraphics[width=\textwidth]{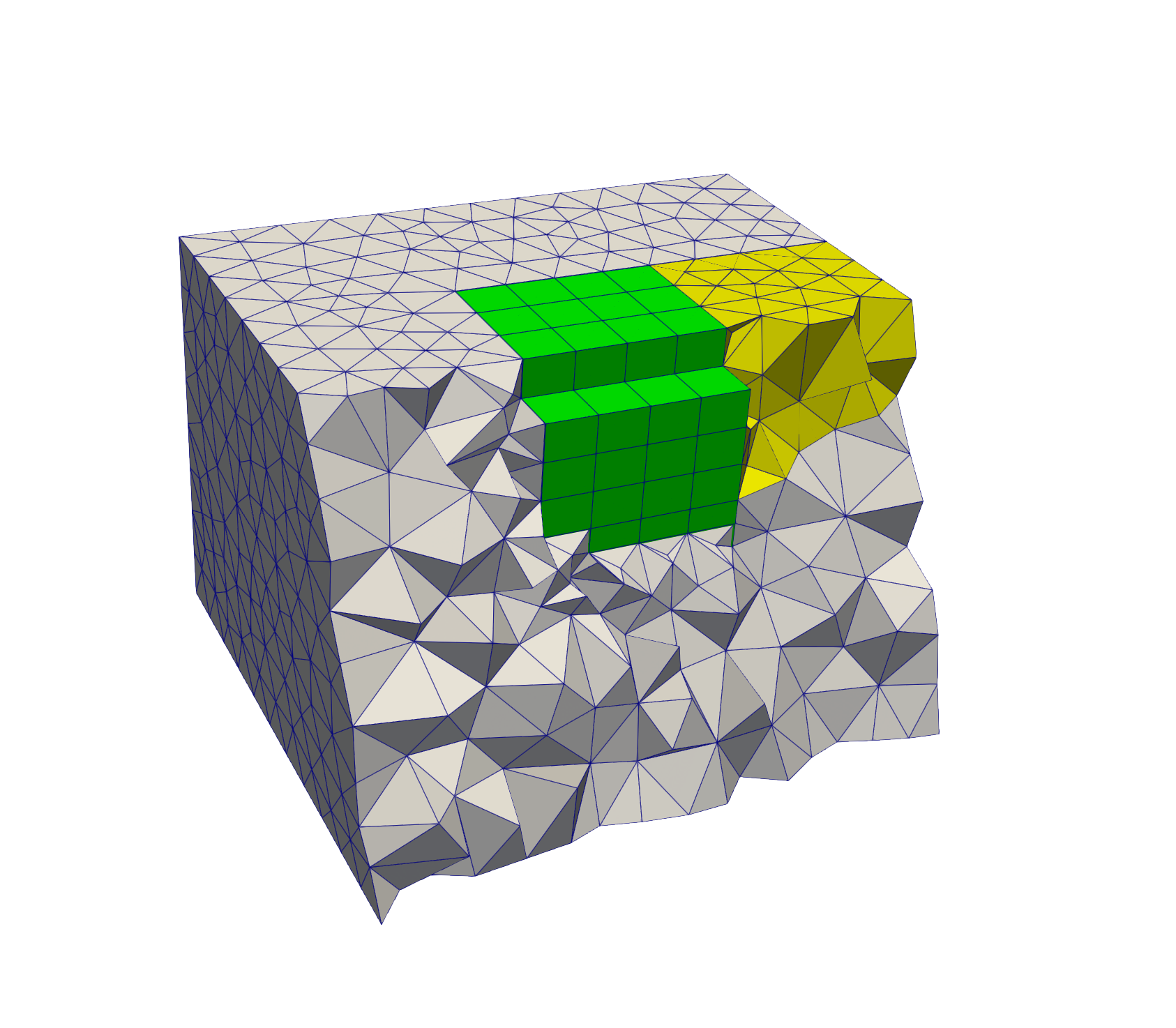}
  \end{minipage}
  \caption{Left: geometry of the cavity (green) inside the porous medium, comprising a wedge (green) and the surrounding box (shadow). Right: example of mesh used in the simulations.}
  \label{fig:cavity-domain}
\end{figure}

\begin{table}\centering
  \begin{tabular}{c|c|c|c|c|c}
    \toprule
    Mesh index & 1 & 2 & 3 & 4 & 5 \\
    \midrule
    Mesh size $h$ & 0.95 & 0.61 & 0.54 & 0.22 & 0.17\\
    $\min_{T\in\Th}h_T$ & 0.31 & 0.15 & 0.15 & 0.05 & 0.05\\
    Num. of elements & 1,326 & 5,935 & 7,963 & 99,748 & 201,653\\
    \bottomrule
  \end{tabular}
  \caption{Characteristics of the mesh family for the tests in Section \ref{sec:lid.cavity}.\label{tab:meshes}}
\end{table}

The forcing term $\bvec{f}=(0,0,-0.98)$ represents the gravity, while we fix $g=0$. The boundary conditions on the velocity are $\bvec{u}(x,y,z)=(x(1-x),0,0)$ on top of the cavity, and $\bvec{u}=\bvec{0}$ elsewhere. Figure \ref{fig:cavity-stream} displays the streamlines obtained on the third mesh in the family with $k=2$. These streamlines show the usual form of circulation inside the cavity for a pure Stokes lid-driven cavity, which drives some (slower) motion inside the wedge section of the porous medium; given the very low permeability of the rest of the medium, little material is transferred into this medium, in which the velocity remains almost zero; in the region $z<-1$ below the cavity, for example, the maximum of the vertex values (obtained by averaging the potential reconstructions in each element surrounding the vertices) of the velocity is below $6\times 10^{-5}$.

\begin{figure}\centering
    \includegraphics[width=.75\textwidth]{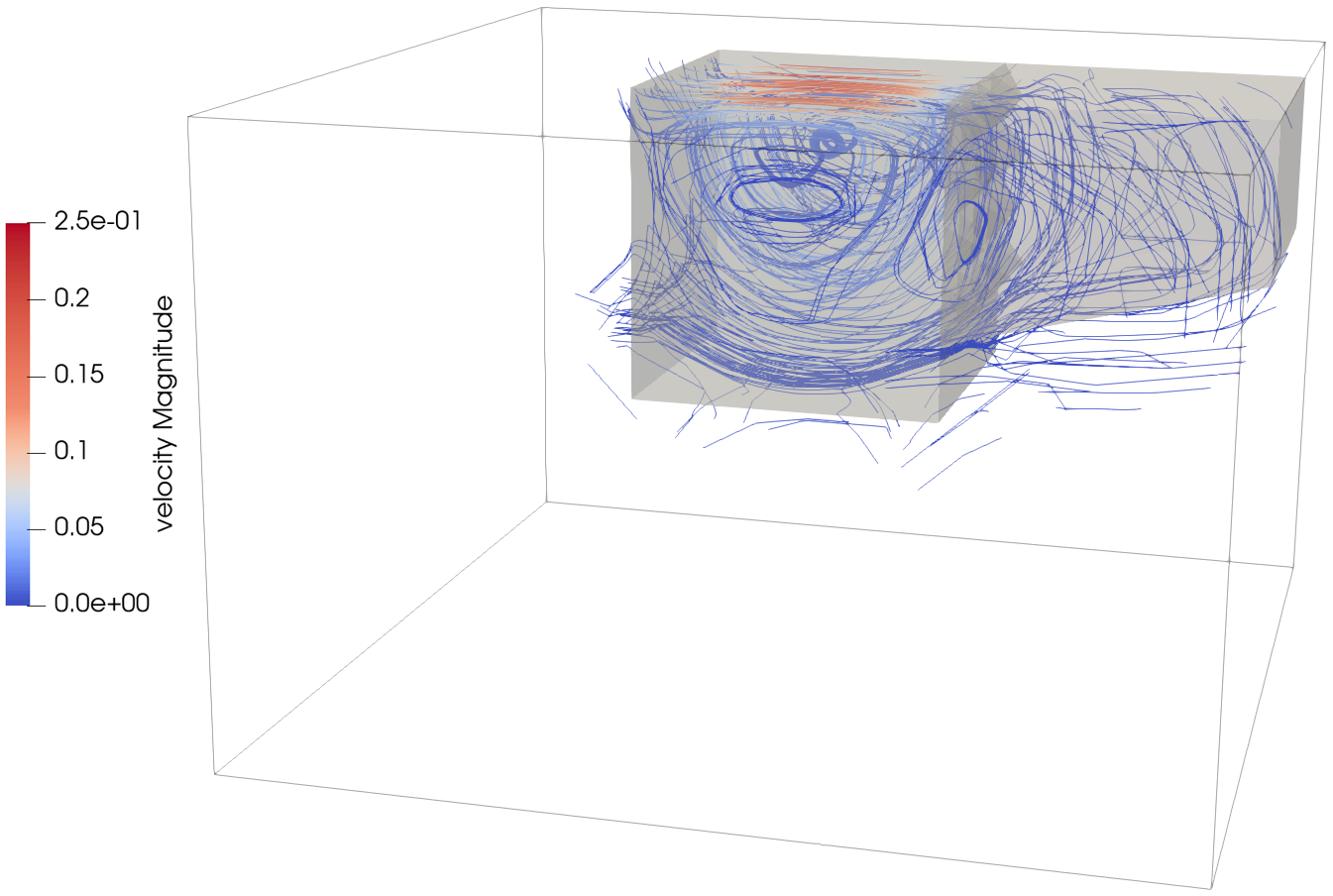}
  \caption{Streamlines for the test case of Section \ref{sec:lid.cavity} (cavity and wedge displayed in shadow).}
  \label{fig:cavity-stream}
\end{figure}

To qualitatively assess the impact of increasing the degree of approximation $k$ of the method, we evaluate for various meshes and degrees the flux across the interface $\Gamma=\{1\}\times (0,1)\times (-0.75,0)$ between the cavity and the wedge. All the meshes $\Mh$ we consider are compatible with this interface, that is, setting $\Gamma_h=\{F\in\Fh\,:\,F\subset \Gamma\}$ we have $\overline{\Gamma} = \bigcup_{F\in\Gamma_h} \overline{F}$.
We then consider the numerical convergence of the numerical flux defined by
\[
\sum_{F\in\Gamma_h}\int_F \bvec{u}_F\cdot\normal_{\Gamma},
\]
where $\normal_\Gamma=(1,0,0)$ is the unit normal to $\Gamma$ pointing inside the wedge.
The values of this flux for different degrees of approximations $k$ are provided in Figure \ref{fig:conv.flux} (left: w.r.t.\ the mesh size; right: w.r.t.\ the total wall time, including assembly and solution time -- notice that the \texttt{HArDCore} library uses multi-threading processes). 
These results show that the lowest order of approximation struggles to provide what seems to be a correct value of the flux, and that the mesh must be extremely fine to get close to this value; on the contrary, for $k\ge 1$, all results, even on coarse meshes and with a low computational cost, seem to  be very close to a given value, indicating that convergence has already occurred. These results corroborate a conclusion already highlighted in \cite{Anderson.Droniou:18}: even on a problem where the solution is not expected to be very regular, slightly increasing the order of approximation of the scheme (here, going from $k=0$ to $k=1$) can lead to a vastly improved accuracy of the numerical outputs at a very low computational cost.

\begin{figure}[!h]\centering
  \ref{br.flux.conv}
  \vspace{0.50cm}\\
  \begin{minipage}{0.45\textwidth}
    \begin{tikzpicture}[scale=0.85]
      \begin{axis} [legend columns=4, legend to name=br.flux.conv]  
        \addplot [mark=star, red] table[x=MeshSize,y=Flux] {dat/cavity/CubeCavityWedge-tets-hexas_k0/flux_file.dat};
        \addlegendentry{$k=0$;}
        \addplot [mark=*, blue] table[x=MeshSize,y=Flux] {dat/cavity/CubeCavityWedge-tets-hexas_k1/flux_file.dat};
        \addlegendentry{$k=1$;}
        \addplot [mark=o, olive] table[x=MeshSize,y=Flux] {dat/cavity/CubeCavityWedge-tets-hexas_k2/flux_file.dat};
        \addlegendentry{$k=2$}
        \addplot [mark=o, black] table[x=MeshSize,y=Flux] {dat/cavity/CubeCavityWedge-tets-hexas_k3/flux_file.dat};
        \addlegendentry{$k=3$}
      \end{axis}            
    \end{tikzpicture}
    \subcaption{w.r.t.~mesh size}
  \end{minipage}
  \begin{minipage}{0.45\textwidth}
    \begin{tikzpicture}[scale=0.85]
      \begin{loglogaxis} 
        \addplot [mark=star, red] table[x=TwallTotal,y=Flux] {dat/cavity/CubeCavityWedge-tets-hexas_k0/flux_file.dat};
        \addplot [mark=*, blue] table[x=TwallTotal,y=Flux] {dat/cavity/CubeCavityWedge-tets-hexas_k1/flux_file.dat};
        \addplot [mark=o, olive] table[x=TwallTotal,y=Flux] {dat/cavity/CubeCavityWedge-tets-hexas_k2/flux_file.dat};
        \addplot [mark=o, black] table[x=TwallTotal,y=Flux] {dat/cavity/CubeCavityWedge-tets-hexas_k3/flux_file.dat};
      \end{loglogaxis}            
    \end{tikzpicture}
    \subcaption{w.r.t.~wall time (seconds)}
  \end{minipage}
  \caption{Convergence of flux values from the cavity to the wedge.}
  \label{fig:conv.flux}
\end{figure}
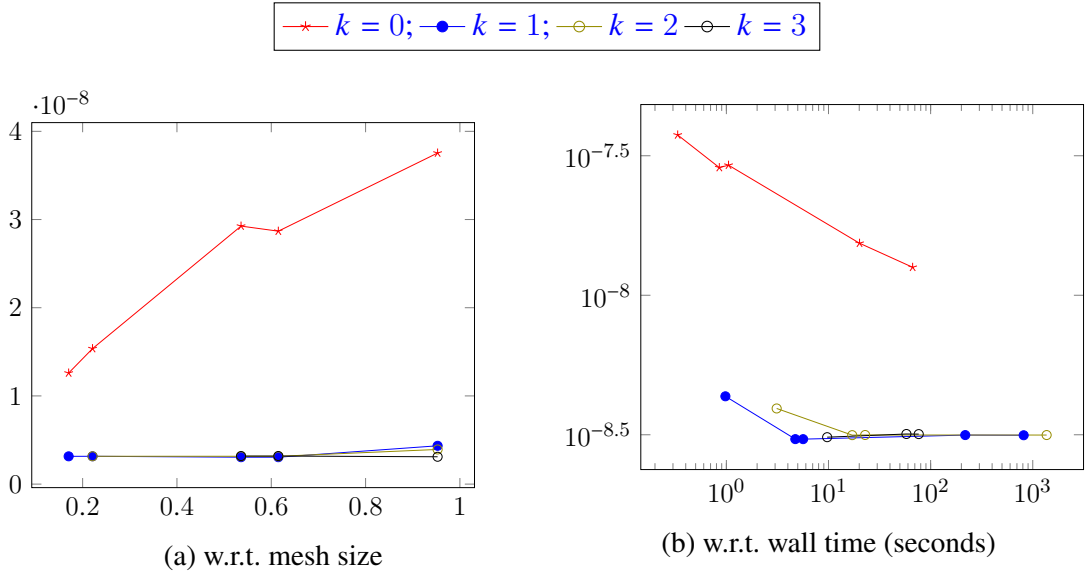


\section{Analysis}\label{sec:analysis}

\subsection{Stability}\label{sec:analysis:stability}

\begin{proposition}[$\norm{\mu,\nu,h}{{\cdot}}$-boundedness of the interpolator]
  With $\beta$ as in Lemma \ref{lem:inf-sup}, it holds, for all $\bvec{v}\in\bvec{H}^1(\Omega;\Real^d)$,
  \begin{equation}\label{eq:boundedness:IT:norm.mu.nu}
    \beta\norm{\mu,\nu,h}{\Ih\bvec{v}}\lesssim\norm{\bvec{H}^1(\Omega;\Real^d)}{\bvec{v}}.
  \end{equation}
\end{proposition}

\begin{proof}
  It holds, by definition,
  $
  \norm{\mu,\nu,h}{\Ih\bvec{v}}^2
  = \sum_{T\in\Th}\left[
  \mu_T\term_1(T) + \nu_T\term_2(T) + \nu_T\term_3(T)
  \right]
  $
  with
  \[
  \begin{gathered}
    \term_1(T) \coloneq
    \norm{\btens{L}^2(T;\Real^{d\times d})}{\GT\IT\bvec{v}}^2
    + \frac{\min(1,\Cf{T}^{-1})}{h_T^2}\norm{\bvec{U},T}{\IT(\bvec{v} - \PST\IT\bvec{v})}^2,
    \\
    \term_2(T) \coloneq
    \norm{\bvec{L}^2(T;\Real^d)}{\tPDT\IT\bvec{v}}^2,\qquad
    \term_3(T) \coloneq
    \min(1,\Cf{T})\norm{\bvec{U},T}{\IT(\bvec{v} - \PDT\IT\bvec{v})}^2.
  \end{gathered}
  \]
  For the first term, combining \eqref{eq:seminorm.1.T:upper.bound} and the fact that the right-hand side of this expression written for $\uvec{v}_T = \IT\bvec{v}$ is $\lesssim \seminorm{\bvec{H}^1(T;\Real^d)}{\bvec{v}}$ by \cite[Eq.~(8.25)]{Di-Pietro.Droniou:20}, we obtain $\term_1(T) \lesssim\seminorm{\bvec{H}^1(T;\Real^d)}{\bvec{v}}^2$.
  For the second term, if $\Cf{T}<1$, we can write $\term_2(T)= \norm{\bvec{L}^2(T;\Real^d)}{\vlproj{k}{T}\bvec{v}}^2\le\norm{\bvec{L}^2(T;\Real^d)}{\bvec{v}}^2$ using the boundedness of $\vlproj{k}{T}$, while, if $\Cf{T}\ge 1$,
  \eqref{eq:approximation:PDT:condition.m=0} gives $\term_2(T)\lesssim\norm{\bvec{L}^2(T;\Real^d)}{\bvec{v}}^2 + h_T^2\seminorm{\bvec{H}^1(T;\Real^d)}{\bvec{v}}^2\le\norm{\bvec{H}^1(T;\Real^d)}{\bvec{v}}^2$, where the conclusion follows observing that $h_T\le 1$ since $\Omega$ has unit diameter by assumption.
  Finally, for the third term, using $\min(1,\Cf{T})\le 1$ and invoking the boundedness \eqref{eq:boundedness:IT:norm.U} of the interpolator in the $\norm{\bvec{U},T}{{\cdot}}$-norm followed by the approximation properties \eqref{eq:approximation:PDT} of $\PDT\IT$ with $(r,m) = (0,0)$ and $(r,m)=(0,1)$ yields
  \[
  \term_3(T)
  \lesssim\norm{\bvec{L}^2(T;\Real^d)}{\bvec{v} - \PDT\IT\bvec{v}}^2
  + h_T^2\seminorm{\bvec{H}^1(T;\Real^d)}{\bvec{v} - \PDT\IT\bvec{v}}^2
  \lesssim
  h_T^2\seminorm{\bvec{H}^1(T;\Real^d)}{\bvec{v}}^2.
  \]
  Gathering the above estimates and recalling the bounds \eqref{eq:mu.nu:bounds} on $\mu$ and $\nu$, the result follows.
\end{proof}

\begin{proof}[Proof of Lemma~\ref{lem:inf-sup}]
  Classical consequence of the continuous inf-sup condition for the divergence $\DIV:\bvec{H}^1_0(\Omega;\Real^d)\to L^2_0(\Omega)$ (see, e.g., \cite{Girault.Raviart:86,Bogovskii:80,Solonnikov:01,Duran.Muschietti:01}) along with the Fortin properties for the interpolator corresponding to \eqref{eq:consistency:bh} and \eqref{eq:boundedness:IT:norm.mu.nu};
  see, e.g., \cite[Section~5.4.3]{Boffi.Brezzi.ea:13} for further details.
\end{proof}

\subsection{Convergence}\label{sec:analysis:convergence}

The purpose of this section is to prove Theorem \ref{thm:error.estimate}.
The proof rests on consistency results for the Stokes, Darcy, and coupling bilinear forms as well as the forcing  term linear form which make the object of the following subsections.

\subsubsection{Consistency of the Stokes bilinear form}

\begin{lemma}[Consistency of the Stokes bilinear form]
  Given $\bvec{w}\in\bvec{H}_0^1(\Omega;\Real^d)$ such that ${\VDIV(\mu\GRAD\bvec{w})}\in\bvec{L}^2(\Omega;\Real^d)$, let the Stokes consistency error linear form $\Errsto(\bvec{w};\cdot):\UhD\to\Real$ be such that, for all $\uvec{v}_h\in\UhD$,
  \begin{equation}\label{eq:Errsto}
    \Errsto(\bvec{w};\uvec{v}_h)
    \coloneq
    -\sum_{T\in\Th}\int_T\VDIV(\mu_T\GRAD\bvec{w})\cdot\bvec{v}_T
    - a_{\mu,h}(\Ih\bvec{w},\uvec{v}_h).
  \end{equation}
  Then, further assuming $\bvec{w}\in\bvec{H}^{r+2}(\Th;\Real^d)$ for some $r\in\{0,\ldots,k\}$, it holds
  \begin{equation}\label{eq:estimate:Errsto}
    \norm{\mu,\nu,h,*}{\Errsto(\bvec{w};\cdot)}
    \lesssim\left(
    \sum_{T\in\Th}\mu_T\min(1,\Cf{T}^{-1})h_T^{2(r+1)}\seminorm{\bvec{H}^{r+2}(T;\Real^d)}{\bvec{w}}^2
    \right)^{\nicefrac12}.
  \end{equation}
\end{lemma}

\begin{proof}
  Let $\uvec{v}_h\in\UhD\setminus\{\uvec{0}\}$.
  Proceeding as in \cite[Point~(ii) in Lemma~2.18]{Di-Pietro.Droniou:20} using an integration by parts for the first term in the definition of $\Errsto$ along with the definitions \eqref{eq:a.mu.h} of $a_{\mu,h}$ and \eqref{eq:GT} of $\GT$ for the second term, we get the following reformulation of the error:
  \[
  \begin{aligned}
    \Errsto(\bvec{w};\uvec{v}_h)
    &= \sum_{T\in\Th}\sum_{F\in\FT}\omega_{TF}\int_F\mu_T(\GRAD\bvec{w} - \GT\IT\bvec{w})\normal_F\cdot(\bvec{v}_F - \bvec{v}_T)
    \\
    &\quad
    - \sum_{T\in\Th}\frac{\mu_T\min(1,\Cf{T}^{-1})}{h_T^2}
    (\IT(\bvec{w} - \PST\IT\bvec{w}), \uvec{v}_T - \IT\PST\uvec{v}_T)_{\bvec{U},T}.
  \end{aligned}
  \]
  Using Cauchy--Schwarz and H\"older inequalities along with $\norm{\bvec{L}^\infty(F;\Real^d)}{\normal_F}\le 1$ for all $F\in\Fh$, we can write
  \begin{equation}\label{eq:consistency.stokes:basic}
    \Errsto(\bvec{w};\uvec{v}_h)\lesssim\sum_{T\in\Th}\left[
      \term_1(T) + \term_2(T)
      \right]
  \end{equation}
  with
  \[
  \begin{aligned}
    \term_1(T)&\coloneq
    \mu_T^{\nicefrac12} h_T^{\nicefrac12}\norm{\btens{L}^2(\partial T;\Real^{d\times d})}{\GRAD\bvec{w} - \GT\IT\bvec{w}}
    ~\left(
    \frac{\mu_T}{h_T}\sum_{F\in\FT}\norm{\bvec{L}^2(F;\Real^d)}{\bvec{v}_F - \bvec{v}_T}^2
    \right)^{\nicefrac12},
    \\
    \term_2(T)&\coloneq
    \frac{\mu_T\min(1,\Cf{T}^{-1})}{h_T^2}
    \norm{\bvec{U},T}{\IT(\bvec{w} - \PST\IT\bvec{w})}
    \norm{\bvec{U},T}{\uvec{v}_T - \IT\PST\uvec{v}_T}.
  \end{aligned}
  \]
  
  Let us estimate $\term_1(T)$.
  Recalling that $\GT\IT\bvec{w} = \tlproj{k}{T}\GRAD\bvec{w}$ and using the approximation properties of the $L^2$-orthogonal projector (cf.~\cite{Di-Pietro.Droniou:17} and \cite[Chapter~1]{Di-Pietro.Droniou:20} concerning the extension to non-star-shaped elements), it is readily inferred for the first factor
  \begin{equation}\label{eq:est.err.grad}
    \mu_T^{\nicefrac12} h_T^{\nicefrac12}\norm{\btens{L}^2(\partial T;\Real^{d\times d})}{\GRAD\bvec{w} - \GT\IT\bvec{w}}
    \lesssim
    \mu_T^{\nicefrac12}h_T^{r+1}\seminorm{\bvec{H}^{r+2}(T;\Real^d)}{\bvec{w}}.
  \end{equation}
  The estimate of the second factor depends on the regime.
  If $\Cf{T}<1$, using \eqref{eq:seminorm.1.T:lower.bound} we write
  \begin{equation}\label{eq:est.seminorm.1.pT:CfT<1}
    \frac{\mu_T}{h_T}\sum_{F\in\FT}\norm{\bvec{L}^2(F;\Real^d)}{\bvec{v}_F - \bvec{v}_T}^2
    \lesssim\mu_T\norm{\stokes,T}{\uvec{v}_T}^2
    = \mu_T\min(1,\Cf{T}^{-1})\norm{\stokes,T}{\uvec{v}_T}^2,
  \end{equation}
  where the conclusion follows observing that $1 = \min(1,\Cf{T}^{-1})$.
  If, on the other hand, $\Cf{T}\ge 1$ (which implies, in particular, $\nu_T>0$), we
  split the sum separating the contributions from internal and boundary faces:
  \begin{equation}\label{eq:bnd.term:split}
    \frac{\mu_T}{h_T}\sum_{F\in\FT}\norm{\bvec{L}^2(F;\Real^d)}{\bvec{v}_F - \bvec{v}_T}^2
    = \frac{\mu_T}{h_T}\sum_{F\in\FT\setminus\Fhb}\norm{\bvec{L}^2(F;\Real^d)}{\bvec{v}_F - \bvec{v}_T}^2
    + \frac{\mu_T}{h_T}\sum_{F\in\FT\cap\Fhb}\norm{\bvec{L}^2(F;\Real^d)}{\bvec{v}_T}^2,
  \end{equation}
  where we have additionally accounted for the fact that $\bvec{v}_F = \bvec{0}$ whenever $F\in\FT\cap\Fhb$ since $\uvec{v}_h\in\UhD$.
  For $F\in\FT\setminus\Fhb$, we insert $\pm\PDT\uvec{v}_T$ into the norm and use triangle and discrete trace inequalities to write
  \begin{equation}\label{eq:bnd.term:est.internal}
    \begin{aligned}
      \frac{\mu_T}{h_T}\norm{\bvec{L}^2(F;\Real^d)}{\bvec{v}_F - \bvec{v}_T}^2
      &\lesssim
      \nu_T\Cf{T}^{-1}\left(
      h_T\norm{\bvec{L}^2(F;\Real^d)}{\bvec{v}_F - \PDT\uvec{v}_T}^2
      + \norm{\bvec{L}^2(T;\Real^d)}{\bvec{v}_T - \PDT\uvec{v}_T}^2
      \right)
      \\
      &\lesssim
      \nu_T\Cf{T}^{-1}
      \norm{\bvec{U},T}{\uvec{v}_T - \IT\PDT\uvec{v}_T}^2\\
      &\lesssim
      \nu_T\Cf{T}^{-1}
      \norm{\darcy,T}{\uvec{v}_T}^2,
      \end{aligned}
    \end{equation}
    where we have additionally used the definition \eqref{eq:CfT} of $\Cf{T}$ in the first inequality,
    invoked the definition of $\norm{\bvec{U},T}{{\cdot}}$ to pass to the second inequality (see \eqref{eq:prod.U}--\eqref{eq:norm.U}, and notice that the term corresponding to $F$ appears in this norm since $F\not\in\Fhb$),
      and concluded using the definition \eqref{eq:norm.nu.darcy} of the $\norm{\darcy,T}{{\cdot}}$-norm together with $1 = \min(1,\Cf{T})$.
      For $F\in\FT\cap\Fhb$, on the other hand, using a discrete trace inequality to write $\norm{\bvec{L}^2(F;\Real^d)}{\bvec{v}_T}\lesssim h_T^{-\nicefrac12}\norm{\bvec{L}^2(T;\Real^d)}{\bvec{v}_T}$,
      inserting $\pm\PDT\uvec{v}_T$ into the norm in the right-hand side, and concluding with a triangle inequality along with the definition of $\norm{\bvec{U},T}{{\cdot}}$, we get
    \begin{equation}\label{eq:bnd.term:est.bnd}
      \frac{\mu_T}{h_T}\norm{\bvec{L}^2(F;\Real^d)}{\bvec{v}_T}^2
      \lesssim\nu_T\Cf{T}^{-1}\left(
      \norm{\bvec{L}^2(T;\Real^d)}{\PDT\uvec{v}_T}^2+\norm{\bvec{U},T}{\uvec{v}_T - \IT\PDT\uvec{v}_T}^2
      \right)
      \lesssim\nu_T\Cf{T}^{-1}\norm{\darcy,T}{\uvec{v}_T}^2,
    \end{equation}
    where the last passage follows recalling the definitions \eqref{eq:norm.nu.darcy} of the $\norm{\darcy,T}{{\cdot}}$-norm, \eqref{eq:tPDT} of $\tPDT$ (which is equal to $\PDT$ since $\Cf{T}\ge 1$), and observing again that $1 = \min(1,\Cf{T})$.
    Hence, plugging \eqref{eq:bnd.term:est.internal} and \eqref{eq:bnd.term:est.bnd} into \eqref{eq:bnd.term:split}, using $\card(\FT)\lesssim 1$, and observing that $\Cf{T}^{-1} = \min(1,\Cf{T}^{-1})$, we can go on writing
  \begin{equation}\label{eq:est.seminorm.1.pT:CfT>=1}
      \frac{\mu_T}{h_T}\sum_{F\in\FT}\norm{\bvec{L}^2(F;\Real^d)}{\bvec{v}_F - \bvec{v}_T}^2
      \lesssim\nu_T\min(1,\Cf{T}^{-1})\norm{\darcy,T}{\uvec{v}_T}^2.
    \end{equation}
  Gathering \eqref{eq:est.err.grad}, \eqref{eq:est.seminorm.1.pT:CfT<1}, and \eqref{eq:est.seminorm.1.pT:CfT>=1}, we arrive at
  \begin{equation}\label{eq:consistency.stokes:T1}
    \term_1(T)
    \lesssim\mu_T^{\nicefrac12}\min(1,\Cf{T}^{-1})^{\nicefrac12}h_T^{r+1}\seminorm{\bvec{H}^{r+2}(T;\Real^d)}{\bvec{w}}\left(
    \mu_T\norm{\stokes,T}{\uvec{v}_T}^2
    + \nu_T\norm{\darcy,T}{\uvec{v}_T}^2
    \right)^{\nicefrac12}.
  \end{equation}

  Moving to $\term_2(T)$, using the $\norm{\bvec{U},T}{{\cdot}}$-boundedness \eqref{eq:boundedness:IT:norm.U} of $\IT$ followed by the approximation properties of $\PST\IT$ (consequence, for each of its components, of \cite[Eq.~(2.14) and Theorem~1.48]{Di-Pietro.Droniou:20}), we have
  \[
  \norm{\bvec{U},T}{\IT(\bvec{w} - \PST\IT\uvec{w})}  
  \lesssim\norm{\bvec{L}^2(T;\Real^d)}{\bvec{w} - \PST\IT\uvec{w}}
  + h_T\seminorm{\bvec{H}^1(T;\Real^d)}{\bvec{w} - \PST\IT\uvec{w}}
  \lesssim h_T^{r+2}\seminorm{\bvec{H}^{r+2}(T;\Real^d)}{\bvec{w}}.
  \]
  Plugging this estimate into the definition of $\term_2(T)$ and recalling the definition \eqref{eq:norm.mu.stokes} of $\norm{\stokes,T}{{\cdot}}$, we get
  \begin{equation}\label{eq:consistency.stokes:T2}
    \term_2(T)
    \lesssim \mu_T^{\nicefrac12}\min(1,\Cf{T}^{-1})^{\nicefrac12}h_T^{r+1}\seminorm{\bvec{H}^{r+2}(T;\Real^d)}{\bvec{w}}~
    \mu_T^{\nicefrac12}\norm{\stokes,T}{\uvec{v}_T}.
  \end{equation}

  Using \eqref{eq:consistency.stokes:T1} and \eqref{eq:consistency.stokes:T2} to estimate the right-hand side of \eqref{eq:consistency.stokes:basic}, we obtain
  \[
  \begin{aligned}
    \Errsto(\bvec{w};\uvec{v}_h)
    &\lesssim\sum_{T\in\Th}\mu_T^{\nicefrac12}\min(1,\Cf{T}^{-1})^{\nicefrac12}h_T^{r+1}\seminorm{\bvec{H}^{r+2}(T;\Real^d)}{\bvec{w}}\left(
    \mu_T\norm{\stokes,T}{\uvec{v}_T}^2
    + \nu_T\norm{\darcy,T}{\uvec{v}_T}^2
    \right)^{\nicefrac12}
    \\
    &\le
    \left(
    \sum_{T\in\Th}\mu_T\min(1,\Cf{T}^{-1})h_T^{2(r+1)}\seminorm{\bvec{H}^{r+2}(T;\Real^d)}{\bvec{w}}^2
    \right)^{\nicefrac12}\norm{\mu,\nu,h}{\uvec{v}_h},
  \end{aligned}
  \]
  where the conclusion follows using a discrete Cauchy--Schwarz inequality on the sum over $T\in\Th$ along with the definition \eqref{eq:norm.mu.nu.h} of $\norm{\mu,\nu,h}{{\cdot}}$.
  Dividing by $\norm{\mu,\nu,h}{\uvec{v}_h}$ and passing to the supremum concludes the proof of \eqref{eq:estimate:Errsto}.
\end{proof}

\subsubsection{Consistency of the Darcy bilinear form}

\begin{lemma}[Consistency of the Darcy bilinear form]
  Given $\bvec{w}\in\bvec{H}^1(\Omega;\Real^d)$, let the Darcy consistency error linear form $\Errdar(\bvec{w};\cdot):\Uh\to\Real$ be such that, for all $\uvec{v}_h\in\Uh$,
  \begin{equation}\label{eq:Errdar}
    \Errdar(\bvec{w};\uvec{v}_h)
    \coloneq
    \sum_{T\in\Th}\int_T\nu_T\bvec{w}\cdot\tPDT\uvec{v}_T
    - a_{\nu,h}(\Ih\bvec{w},\uvec{v}_h).
  \end{equation}
  Then, further assuming $\bvec{w}\in\bvec{H}^{r+1}(\Th;\Real^d)$ for some $r\in\{0,\ldots,k\}$, it holds
  \begin{equation}\label{eq:estimate:Errdar}
    \norm{\mu,\nu,h,*}{\Errdar(\bvec{w};\cdot)}
    \lesssim\left(
    \sum_{T\in\Th}\nu_T\min(1,\Cf{T}) h_T^{2(r+1)}\seminorm{\bvec{H}^{r+1}(T;\Real^d)}{\bvec{w}}^2
    \right)^{\nicefrac12}.
  \end{equation}
\end{lemma}

\begin{proof}
  Let $\uvec{v}_h\in\UhD\setminus\{\uvec{0}\}$.
  Expanding $a_{\nu,h}$ according to its definition \eqref{eq:a.nu}, we get
  \begin{equation}\label{eq:estimate:Errdar:basic}
    \Errdar(\bvec{w};\uvec{v}_h)
    = \sum_{T\in\Th}\left[
    \term_1(T) + \term_2(T)
    \right],
  \end{equation}
  with
  \[
  \begin{aligned}
    \term_1(T) &\coloneq
    \int_T\nu_T(\bvec{w} - \tPDT\IT\bvec{w})\cdot\tPDT\uvec{v}_T,
    \\
    \term_2(T) &\coloneq
    -\nu_T\min(1,\Cf{T})(\IT(\bvec{w} - \PDT\IT\bvec{w}), \uvec{v}_T - \IT\PDT\uvec{v}_T)_{\bvec{U},T}.
  \end{aligned}
  \]

  The estimate of $\term_1(T)$ depends on the regime.
  In the case $\Cf{T}<1$, the definitions \eqref{eq:tPDT} of $\tPDT$ and \eqref{eq:def.Ih} of the interpolator yield
  \begin{equation}\label{eq:est.T1.Cf<1}
    \term_1(T)=\int_T\nu_T(\bvec{w} - \vlproj{k}{T}\bvec{w})\cdot\bvec{v}_T=0
  \end{equation}
  since $\nu_T\bvec{v}_T\in\vPoly{k}(T;\Real^d)$.
  Let us now consider the case $\Cf{T}\ge 1$.
  Using again the definition \eqref{eq:tPDT} to replace $\tPDT$ with $\PDT$ and applying a Cauchy--Schwarz inequality, we get
  \begin{equation}\label{eq:est.T1.Cf>1}
    \begin{aligned}
      |\term_1(T)|
      &\lesssim\nu_T\norm{\bvec{L}^2(T;\Real^d)}{\bvec{w} - \PDT\IT\bvec{w}}
      \norm{\bvec{L}^2(T;\Real^d)}{\PDT\uvec{v}_T}
      \\
      &\lesssim\nu_T^{\nicefrac12}\min(1,\Cf{T})^{\nicefrac12}h_T^{r+1}\seminorm{\bvec{H}^{r+1}(T;\Real^d)}{\bvec{w}}~\nu_T^{\nicefrac12}\norm{\darcy,T}{\uvec{v}_T},
    \end{aligned}
  \end{equation}
  where, to pass to the second line, we have used the approximation properties \eqref{eq:approximation:PDT} of $\PDT\IT$ with $m=0$,
  the definition \eqref{eq:norm.nu.darcy} of the $\norm{\darcy,T}{{\cdot}}$-norm,
  and observed that $1 = \min(1,\Cf{T})$.  
  Gathering \eqref{eq:est.T1.Cf<1} and \eqref{eq:est.T1.Cf>1}, we thus have, for any value of $\Cf{T}$,
  \begin{equation}\label{eq:estimate:Errdar:T1}
    |\term_1(T)|
    \lesssim
    \nu_T^{\nicefrac12}\min(1,\Cf{T})^{\nicefrac12}h_T^{r+1}\seminorm{\bvec{H}^{r+1}(T;\Real^d)}{\bvec{w}}
    \nu_T^{\nicefrac12}\norm{\darcy,T}{\uvec{v}_T}.
  \end{equation}
  
  To estimate $\term_2(T)$, we use a Cauchy--Schwarz inequality to write
  \begin{align}
      &|\term_2(T)|
      \nonumber\\
      &\quad
      \le\nu_T^{\nicefrac12}\min(1,\Cf{T})^{\nicefrac12}\norm{\bvec{U},T}{\IT(\bvec{w} - \PDT\IT\bvec{w})}
      ~\nu_T^{\nicefrac12}\min(1,\Cf{T})^{\nicefrac12}\norm{\bvec{U},T}{\uvec{v}_T - \IT\PDT\uvec{v}_T}
      \nonumber\\
      &\quad
      \lesssim
      \nu_T^{\nicefrac12}\min(1,\Cf{T})^{\nicefrac12}\left(
      \norm{\bvec{L}^2(T;\Real^d)}{\bvec{w} - \PDT\IT\bvec{w}}
      + h_T\seminorm{\bvec{H}^1(T;\Real^d)}{\bvec{w} - \PDT\IT\bvec{w}}
      \right)
      ~\nu_T^{\nicefrac12}\norm{\darcy,T}{\uvec{v}_T}
      \nonumber\\
      &\quad
      \lesssim
      \nu_T^{\nicefrac12}\min(1,\Cf{T})^{\nicefrac12}h_T^{r+1}\seminorm{\bvec{H}^{r+1}(T;\Real^d)}{\bvec{w}}
      ~\nu_T^{\nicefrac12}\norm{\darcy,T}{\uvec{v}_T},
    \label{eq:estimate:Errdar:T2}
  \end{align}
  where we have used the $\norm{\bvec{U},T}{{\cdot}}$-boundedness \eqref{eq:boundedness:IT:norm.U} of $\IT$ along with the definition \eqref{eq:norm.nu.darcy} of the $\norm{\darcy,T}{{\cdot}}$-norm in the second inequality and
  the approximation properties \eqref{eq:approximation:PDT} of $\PDT\IT$ with $m = 0$ and $m = 1$ to conclude.
  Plugging \eqref{eq:estimate:Errdar:T1} and \eqref{eq:estimate:Errdar:T2} into \eqref{eq:estimate:Errdar:basic}, using discrete Cauchy--Schwarz inequalities, dividing by $\norm{\mu,\nu,h}{\uvec{v}_h}$, and passing to the supremum, the conclusion follows.
\end{proof}

\subsubsection{Consistency of the coupling bilinear form}

The quantity estimated in the following lemma can be interpreted as an adjoint consistency error for the discrete divergence.

\begin{lemma}[Consistency of the coupling bilinear form]
  Given $q\in H^1(\Omega)$, let the coupling consistency error linear form $\Errc(q;\cdot):\UhD\to\Real$ be such that, for all $\uvec{v}_h\in\UhD$,
  \begin{equation}\label{eq:Errc}
    \Errc(q;\uvec{v}_h)
    \coloneq
    \sum_{T\in\Th}\int_T\GRAD q\cdot\tPDT\uvec{v}_T - b_h(\uvec{v}_h,\lproj{k}{h}q).
  \end{equation}
  Then, further assuming, for some $r\in\{0,\ldots,k\}$, $q\in H^{r+1+\tv{\Cf{T}\ge 1}}(T)$ for all $T\in\Th$, it holds
  \begin{equation}\label{eq:estimate:Errc}
    \norm{\mu,\nu,h,*}{\Errc(q;\cdot)}
    \\
    \lesssim\left[
      \sum_{T\in\Th}\left(
      \mu_T^{-1}\tv{\Cf{T}<1} h_T^{2(r+1)}\seminorm{H^{r+1}(T)}{q}^2
      + \nu_T^{-1}\tv{\Cf{T}\ge 1} h_T^{2(r+1)}\seminorm{H^{r+2}(T)}{q}^2
      \right)
      \right]^{\nicefrac12},
  \end{equation}
  where $\nu_T^{-1}\tv{\Cf{T}\ge 1}\coloneq 0$ if $\nu_T = 0$, as in Theorem~\ref{thm:error.estimate}.
\end{lemma}

\begin{proof}
  Let $\uvec{v}_h\in\UhD\setminus\{\uvec{0}\}$.
  We start by noticing that, expanding the bilinear form $b_h$ according to its definition \eqref{eq:bh},
  \begin{equation}\label{eq:Errc:estimate:basic}
    \Errc(q;\uvec{v}_h)
    = \sum_{T\in\Th}
    \left(
    \int_T\GRAD q\cdot\tPDT\uvec{v}_T
    + \int_T\lproj{k}{T}q~\DT\uvec{v}_T
    - \sum_{F\in\FT}\omega_{TF}\int_Fq~(\bvec{v}_F\cdot\normal_F)
    \right),
  \end{equation}
  where the insertion of the last term in parenthesis is made possible by the single-valuedness of $q$ at interfaces along with the fact that $\bvec{v}_F\cdot\normal_F = 0$ for all $F\in\Fhb$.
  Denote by $\term(T)$ the argument of the summation in \eqref{eq:Errc:estimate:basic}.
  To estimate this quantity, we distinguish two cases based on the value of $\Cf{T}$.
  
  If $\Cf{T}<1$, $\tPDT\uvec{v}_T = \bvec{v}_T$ by \eqref{eq:tPDT}, so that
  \[
  \begin{aligned}
    \term(T)
    &= \int_T\GRAD q\cdot\bvec{v}_T
    + \int_T\lproj{k}{T}q~\DT\uvec{v}_T
    - \sum_{F\in\FT}\omega_{TF}\int_Fq~(\bvec{v}_F\cdot\normal_F)
    \\
    &= -\int_T\GRAD(\lproj{k}{T}q - q)\cdot\bvec{v}_T
    + \sum_{F\in\FT}\omega_{TF}\int_F(\lproj{k}{T}q - q)~(\bvec{v}_F\cdot\normal_F)
    \\
    &= \int_T\cancel{(\lproj{k}{T}q - q)}~\DIV\bvec{v}_T
    + \sum_{F\in\FT}\omega_{TF}\int_F(\lproj{k}{T}q - q)~(\bvec{v}_F - \bvec{v}_T)\cdot\normal_F,
  \end{aligned}
  \]
  where, to obtain the second equality, recalling \eqref{eq:DT}, we have expanded $\DT\uvec{v}_T$ according to \eqref{eq:GT} with $\btens{\tau} = \lproj{k}{T}q\Id$,
  while the third equality is obtained integrating by parts the first term in the right-hand side, with the cancellation resulting from the definition of the $L^2$-orthogonal projector along with $\DIV\bvec{v}_T\in\Poly{k-1}(T)\subset\Poly{k}(T)$.
  Using H\"older and Cauchy--Schwarz inequalities along with the fact that $\norm{\bvec{L}^\infty(F;\Real^d)}{\normal_F} \le 1$, we obtain
  \begin{align}
    \term(T)
    \lesssim{}&
    h_T^{\nicefrac12} \norm{L^2(\partial T)}{q - \lproj{k}{T}q}
    \left(
    \frac{1}{h_T}\sum_{F\in\FT}\norm{\bvec{L}^2(F;\Real^d)}{\bvec{v}_T - \bvec{v}_F}^2
    \right)^{\nicefrac12}
    \nonumber\\
    \lesssim{}&
    h_T^{r+1}\seminorm{H^{r+1}(T)}{q}\norm{\stokes,T}{\uvec{v}_T}
    =
    \mu_T^{-\nicefrac12}\tv{\Cf{T}<1}^{\nicefrac12}
    h_T^{r+1}\seminorm{H^{r+1}(T)}{q}~\mu_T^{\nicefrac12}\norm{\stokes,T}{\uvec{v}_T},
    \label{eq:termT.estimate:CfT<1}
  \end{align}
  where we have used
  the approximation properties of $\lproj{k}{T}$ and \eqref{eq:seminorm.1.T:lower.bound} to pass to the second line.

  If $\Cf{T}\ge 1$, on the other hand, we have $\tPDT = \PDT$ (cf.\ \eqref{eq:tPDT}), so that
   \[
   \term(T)
   =
   \int_T\GRAD q\cdot\PDT\uvec{v}_T
   + \int_T\cancel{\lproj{k}{T}}\,q~\DT\uvec{v}_T
   - \sum_{F\in\FT}\omega_{TF}\int_Fq~(\bvec{v}_F\cdot\normal_F),
   \]
   where the cancellation of the projector follows from its definition.
   We next proceed as in \cite[Theorem~11]{Di-Pietro.Droniou:21*1}.
   The definition \eqref{eq:PDT} of $\PDT$ with $(q,\bvec{w})\gets(\lproj{k+1}{T}q,\bvec{0})$ gives
   \[
   \int_T\PDT\uvec{v}_T\cdot\GRAD \lproj{k+1}{T}q + \int_T\DT\uvec{v}_T~\lproj{k+1}{T}q
   - \sum_{F\in\FT}\omega_{TF}\int_F(\bvec{v}_F\cdot\normal_F)~\lproj{k+1}{T}q=0.
   \]
   Subtracting this quantity from $\term(T)$ and rearranging the terms yields
   \begin{multline*}
     \term(T)
     = \int_T\GRAD(q - \lproj{k+1}{T}q)\cdot\PDT\uvec{v}_T
     \\
     + \cancel{\int_T (q - \lproj{k+1}{T}q)~\DT\uvec{v}_T}
     + \sum_{F\in\FT}\omega_{TF}\int_F(\lproj{k+1}{T}q - q)~(\bvec{v}_F\cdot\normal_F),
   \end{multline*}
   where the cancellation comes from the definition of $\lproj{k+1}{T}$ together with $\DT\uvec{v}_T\in\Poly{k}(T)\subset\Poly{k+1}(T)$.
   Applying Cauchy--Schwarz and H\"older inequalities, we go on writing
   \begin{multline*}
     \term(T)
     \lesssim\left(
     \norm{\bvec{L}^2(T;\Real^d)}{\GRAD(q - \lproj{k+1}{T}q)}^2
     + h_T^{-1} \norm{L^2(\partial T)}{q - \lproj{k+1}{T}q}^2
     \right)^{\nicefrac12}
     \\
     \qquad\times
     \left(
     \norm{\bvec{L}^2(T;\Real^2)}{\PDT\uvec{v}_T}^2
     + h_T\sum_{F\in\FT}\norm{\bvec{L}^2(F;\Real^d)}{\bvec{v}_F\cdot\normal_F}^2
     \right)^{\nicefrac12}.
   \end{multline*}
   Finally, using the approximation properties of $\lproj{k+1}{T}$ for the first factor and noticing that the second factor is $\lesssim\norm{\darcy,T}{\uvec{v}_T}$ (this estimate is analogous to the one of the second factor in \cite[Eq.~(6.37)]{Di-Pietro.Droniou:21*1}, and can easily be derived from the definition \eqref{eq:a.darcy.T} of $a_{\darcy,T}$ by triangle and discrete trace inequalities), we get   
   \begin{equation}\label{eq:termT.estimate:CfT>=1}
     \term(T)
     \lesssim h_T^{r+1}\seminorm{H^{r+2}(T)}{q}~\norm{\darcy,T}{\uvec{v}_T}
     \le \nu_T^{-\nicefrac12}\tv{\Cf{T}\ge 1}^{\nicefrac12}
     h_T^{r+1}\seminorm{H^{r+2}(T)}{q}~\nu_T^{\nicefrac12}\norm{\darcy,T}{\uvec{v}_T},
   \end{equation}
   where we have additionally noticed that $\Cf{T}\ge 1$ implies $\nu_T>0$.

   To conclude, we plug \eqref{eq:termT.estimate:CfT<1} and \eqref{eq:termT.estimate:CfT>=1} into \eqref{eq:Errc:estimate:basic},
   use a Cauchy--Schwarz inequality on the sum over $T\in\Th$,
   recall the definition \eqref{eq:norm.mu.nu.h} of the $\norm{\mu,\nu,h}{{\cdot}}$-norm,
  and pass to the supremum after dividing by $\norm{\mu,\nu,h}{\uvec{v}_h}$.
\end{proof}

\begin{remark}[Discretisation of the source term]\label{rem:disc.source}
  The use of $\PDT$ in the discretisation of the source term when $\Cf{T}\ge 1$ (see \eqref{eq:discrete:variational} and \eqref{eq:tPDT}) is crucial to ensure that, in this case, the consistency error of the coupling bilinear form can be bounded from above using the Darcy norm instead of the Stokes norm; compare \eqref{eq:termT.estimate:CfT>=1} and \eqref{eq:termT.estimate:CfT<1}. This bound is  key to establishing an error estimate in $h^{r+1}$ that remains robust in the Darcy limit. 
\end{remark}

\subsubsection{Consistency of the forcing term linear form}

The following lemma estimates the difference between the standard HHO right-hand side linear form and the one obtained, as in \eqref{eq:strong:momentum}, using $\tPDT\uvec{v}_T$ instead of $\bvec{v}_T$ as a test function.

\begin{lemma}[Consistency of the forcing term]
  For any $\bvec{\varphi}\in\bvec{L}^2(\Omega;\Real^d)$, define the right-hand side consistency error linear form $\Errrhs(\bvec{\varphi};\cdot):\Uh\to\Real$ such that, for all $\uvec{v}_h\in\Uh$,
  \begin{equation}\label{eq:Errrhs}
    \Errrhs(\bvec{\varphi};\uvec{v}_h)\coloneq
    \sum_{T\in\Th}\int_T\bvec{\varphi}\cdot(\bvec{v}_T - \tPDT\uvec{v}_T).
  \end{equation}
  Further assuming $\bvec{\varphi}\in\bvec{H}^r(\Th;\Real^d)$ for some $r\in\{0,\ldots,k\}$, it holds
  \begin{equation}\label{eq:estimate:Errrhs}
    \norm{\mu,\nu,h,*}{\Errrhs(\bvec{\varphi};\cdot)}
    \lesssim\left(
    \sum_{T\in\Th}\mu_T^{-1}\min(1,\Cf{T}^{-1}) h_T^{2(r+1)}\seminorm{\bvec{H}^r(T;\Real^d)}{\bvec{\varphi}}^2
    \right)^{\nicefrac12}.
  \end{equation}
\end{lemma}

\begin{proof}
  Denote by $\term(T)$ the argument of the summation in \eqref{eq:Errrhs}.
  If $\Cf{T}<1$, the definition \eqref{eq:tPDT} of $\tPDT$ yields $\term(T) = 0$.
  Consider now the case $\Cf{T}\ge 1$ (which implies, in particular, $\nu_T>0$). We first notice that, letting $\vlproj{k-1}{T}\bvec{\varphi} \coloneq \bvec{0}$ if $k=0$,
  \begin{equation}\label{eq:approx.pik}
    \norm{\bvec{L}^2(T;\Real^d)}{\bvec{\varphi}-\vlproj{k-1}{T}\bvec{\varphi}}\lesssim h_T^r\seminorm{\bvec{H}^r(T;\Real^d)}{\bvec{\varphi}},
  \end{equation}
  where the result is trivial if $k=0$ (which imposes $r=0$) and otherwise follows from the approximation properties of $\vlproj{k-1}{T}$, see \cite[Theorem 1.45]{Di-Pietro.Droniou:20}. Recalling that, for $\Cf{T}\ge 1$, we have $\tPDT = \PDT$ by \eqref{eq:tPDT} and invoking \eqref{eq:lproj.k-1.PDT} (which trivially holds also for $k=0$), we then write
  \[
  \begin{aligned}
    \term(T)
    &= \int_T(\bvec{\varphi} - \vlproj{k-1}{T}\bvec{\varphi})\cdot(\bvec{v}_T - \PDT\uvec{v}_T)
    \\
    &\le\norm{\bvec{L}^2(T;\Real^d)}{\bvec{\varphi} - \vlproj{k-1}{T}\bvec{\varphi}}
    \norm{\bvec{L}^2(T;\Real^d)}{\bvec{v}_T - \PDT\uvec{v}_T}
    \\
    &\lesssim
    \mu_T^{-\nicefrac12} h_T^r \seminorm{\bvec{H}^r(T;\Real^d)}{\bvec{\varphi}}
    ~h_T\mu_T^{\nicefrac12}h_T^{-1}\norm{\darcy,T}{\uvec{v}_T}
    \\
    &=
    \mu_T^{-\nicefrac12} h_T^{r+1} \seminorm{\bvec{H}^r(T;\Real^d)}{\bvec{\varphi}}
    ~\nu_T^{\nicefrac12}\min(1,\Cf{T}^{-1})^{\nicefrac12}\norm{\darcy,T}{\uvec{v}_T},
  \end{aligned}
  \]
  where we have used Cauchy--Schwarz inequalities in the first inequality,
  the approximation properties \eqref{eq:approx.pik} of the $L^2$-orthogonal projector for the first factor together with the definitions \eqref{eq:norm.U} and \eqref{eq:norm.nu.darcy} of $\norm{\bvec{U},T}{{\cdot}}$ and $\norm{\darcy,T}{{\cdot}}$ to write $\norm{\bvec{L}^2(T;\Real^d)}{\bvec{v}_T - \PDT\uvec{v}_T}\le\norm{\bvec{U},T}{\uvec{v}_T-\IT\PDT\uvec{v}_T}\le\norm{\darcy,T}{\uvec{v}_T}$ in the second inequality,
  while the conclusion follows from the definition \eqref{eq:CfT} of $\Cf{T}$ along with $\Cf{T}^{-1} = \min(1,\Cf{T}^{-1})$.
  Using the above estimate in \eqref{eq:Errrhs}, applying a Cauchy--Schwarz inequality on the sum over $T\in\Th$, and recalling the definition \eqref{eq:norm.mu.nu.h} of $\norm{\mu,\nu,h}{{\cdot}}$, \eqref{eq:estimate:Errrhs} follows.
\end{proof}

\subsubsection{Proof of Theorem \ref{thm:error.estimate}}

\begin{proof}[Proof of Theorem \ref{thm:error.estimate}]
  Since $a_{\mu,h}+a_{\nu,h}$ is $1$-coercive and has norm $1$ for the $\norm{\mu,\nu,h}{{\cdot}}$ norm, Lemma~\ref{lem:inf-sup} and \cite[Lemma A.11]{Di-Pietro.Droniou:20} show that $\mathcal{A}_h$ is $\gamma$-inf-sup stable for the norm in the left-hand side of \eqref{eq:error.estimate}. Hence, in the spirit of the third Strang lemma \cite{Di-Pietro.Droniou:18}, this error estimate follows if we bound the consistency error by the bracketed term in the right-hand side.
  The consistency error for the scheme \eqref{eq:discrete:variational} is
  \begin{align}
    \Err(\bvec{u},p;\uvec{v}_h)
    &\coloneq
    \sum_{T\in\Th}\int_T\bvec{f}\cdot\tPDT\uvec{v}_T + \int_\Omega gq_h
    -\mathcal{A}_h((\Ih\bvec{u},\lproj{h}{k}p),(\uvec{v}_h,q_h))    
    \nonumber\\
    &=
    \sum_{T\in\Th}\int_T\VDIV(\mu_T\GRAD\bvec{u})\cdot(\bvec{v}_T - \tPDT\uvec{v}_T)
    -
    \sum_{T\in\Th}\int_T\VDIV(\mu_T\GRAD\bvec{u})\cdot\bvec{v}_T - a_{\mu,h}(\Ih\bvec{u},\uvec{v}_h)
    \nonumber\\
    &\quad
    + \sum_{T\in\Th}\int_T\nu_T\bvec{u}\cdot\tPDT\uvec{v}_T
    - a_{\nu,h}(\Ih\bvec{u},\uvec{v}_h)
    + \sum_{T\in\Th}\int_T\GRAD p\cdot\tPDT\uvec{v}_T
    - b_h(\uvec{v}_h,\lproj{k}{h}p)
    \nonumber\\
    &\quad
    + \cancel{%
      \int_\Omega g q_h
      + b_h(\Ih\bvec{u}, q_h)%
    }
    \nonumber\\
    &=
    \Errrhs(\VDIV(\mu\GRAD\bvec{u});\uvec{v}_h)
    + \Errsto(\bvec{u};\uvec{v}_h)
    + \Errdar(\bvec{u};\uvec{v}_h)
    + \Errc(p;\uvec{v}_h),
  \label{eq:error.estimate:basic}
  \end{align}
  where we have we have replaced $\bvec{f}$ with the left-hand side of \eqref{eq:strong:momentum}, expanded $\mathcal{A}_h$ according to its definition \eqref{eq:Ah}, and used \eqref{eq:consistency:bh} along with \eqref{eq:strong:mass} to cancel the last term in the first passage,
  and concluded using the definitions of the consistency errors, i.e.: \eqref{eq:Errrhs} with $\bvec{\varphi} = \VDIV(\mu\GRAD\bvec{u})$, \eqref{eq:Errsto} and \eqref{eq:Errdar} with $\bvec{w} = \bvec{u}$, and \eqref{eq:Errc} with $q = p$.

  Using, respectively, \eqref{eq:estimate:Errrhs}
  (further noticing that $\seminorm{\bvec{H}^r(T;\Real^d)}{\VDIV(\mu_T\GRAD\bvec{u})}\lesssim \mu_T\seminorm{\bvec{H}^{r+2}(T;\Real^d)}{\bvec{u}}$ for all $T\in\Th$),
  \eqref{eq:estimate:Errsto},
  \eqref{eq:estimate:Errdar},
  and \eqref{eq:estimate:Errc} to estimate the terms in the right-hand side of \eqref{eq:error.estimate:basic}, the result follows.
\end{proof}

\section*{Acknowledgements}

This research received support from the ANR ``NEMESIS'' (ANR-20-MRS2-0004) and the Australian Research Council's Discovery Projects funding scheme (DP210103092). The authors would also like to thank Ricardo Ruiz-Baier for sharing Gmsh geometry files at the source of the tests in Section \ref{sec:lid.cavity}.


\printbibliography

\end{document}